\documentclass[10pt]{amsart}
\usepackage{fouriernc}
\usepackage[english]{babel}
\usepackage{url}

\usepackage[all,2cell,emtex]{xy}
\usepackage{amssymb,amsmath,bbm,xr}
\usepackage[mathscr]{euscript}
\usepackage{mathrsfs}

\usepackage[toc,page]{appendix}
\input cyracc.def
\DeclareFontFamily{U}{russian}{}
\DeclareFontShape{U}{russian}{m}{n}
        { <5><6> wncyr5
        <7><8><9> wncyr7
        <10><10.95><12><14.4><17.28><20.74><24.88> wncyr10 }{}
\DeclareSymbolFont{Russian}{U}{russian}{m}{n}
\DeclareSymbolFontAlphabet{\mathcyr}{Russian}
\makeatletter
\let\@math@cyr\mathcyr
\renewcommand{\mathcyr}[1]{\@math@cyr{\cyracc #1}}
\makeatother

\newcommand{\sch}{\mathit{Sch}}
\newcommand{\reg}{\mathit{Reg}_{k}}
\newcommand{\sft}{\mathscr S^{ft}}

\newcommand{\sm}{\mathit{Sm}}
\newcommand{\smc}{\sm^{\mathit{cor}}}
\newcommand{\sftcx}[1]{\mathscr S^{ft,\mathit{cor}}_{#1}}
\newcommand{\smcx}[1]{\smc_{#1}}

\newcommand{\zequi}{z_{\mathit{equi}}}

\DeclareMathOperator{\sh}{Sh}
\newcommand{\shtr}{\sh^{\mathit{tr}}}
\newcommand{\ushtr}{\underline {\sh}^{\mathit{tr}}}

\newcommand{\eff}{\mathit{eff}}

\newcommand{\DM}{\mathrm{DM}}
\newcommand{\DMe}{\mathrm{DM}^{\eff}}
\newcommand{\uDM}{\underline{\mathrm{DM}\!}\,}
\newcommand{\uDMe}{\underline{\mathrm{DM}\!}\,^{\eff}}

\DeclareMathOperator{\Sp}{Sp}

\DeclareMathOperator{\Hom}{Hom}

\DeclareMathOperator{\uHom}{\underline{Hom}} 

\DeclareMathOperator{\Spec}{Spec}

\newcommand{\ilim} { \varinjlim }

\DeclareMathOperator{\Der}{D}

\newcommand{\derL}{\mathbf{L}}
\newcommand{\derR}{\mathbf{R}}

\newcommand{\NN} {\mathbf N}
\newcommand{\ZZ} {\mathbf Z}
\newcommand{\QQ} {\mathbf Q}

\renewcommand{\AA} {\mathbf A}
\newcommand{\PP} {\mathbf P}

\newcommand{\cO}{\mathcal O}
\newcommand{\T}{\mathcal T}

\newcommand{\un}{\mathbbm 1} 

\newcommand{\uM}{\underline{M\!}\,}

\newcommand{\nis}{{\mathrm{Nis}}}
\newcommand{\et}{\mathrm{\acute{e}t}}
\newcommand{\cdh}{{\mathrm{cdh}}}

\newcommand{\h}  {{\mathrm{h}}}

\newcommand{\SH}{\mathrm{SH}}
\newcommand{\Mod}{\text{-}\mathrm{Mod}}
\newcommand{\HZ}{\mathit{HR}}
\newcommand{\DMcdh}{\DM_{\cdh}}
\newcommand{\ldh}{{\ell\mathrm{dh}}}

\title{Integral mixed motives in equal characteristic}

\author{Denis-Charles Cisinski}
\address{Universit\'e Paul Sabatier\\
Institut de Math\'ematiques de Toulouse\\
118\\ route de Narbonne\\
31062 Toulouse Cedex 9\\France}
\email{denis-charles.cisinski@math.univ-toulouse.fr}
\urladdr{http://www.math.univ-toulouse.fr/~dcisinsk/}
\author{Fr\'ed\'eric D\'eglise}
\address{E.N.S. Lyon\\UMPA\\
46\\all\'ee d'Italie\\
69364 Lyon Cedex~07\\
France}
\email{frederic.deglise@ens-lyon.fr}
\urladdr{http://perso.ens-lyon.fr/frederic.deglise/}
\thanks{Partially supported by the ANR (grant No. ANR-12-BS01-0002)}

\newtheorem{thm}{Theorem}[section]
\newtheorem{prop}[thm]{Proposition}
\newtheorem{lm}[thm]{Lemma}
\newtheorem{cor}[thm]{Corollary}

\theoremstyle{remark} 
\newtheorem{rem}[thm]{Remark}

\newtheorem{ex}[thm]{Example}

\newtheorem*{conj}{Conjecture}

\theoremstyle{definition} 
\newtheorem{df}[thm]{Definition}
\newtheorem{num}[thm]{}
\newtheorem{paragr}[thm]{}

\numberwithin{equation}{thm}

\renewcommand{\leftrightarrows}{\rightleftarrows}

\begin{document}

\begin{abstract}
For noetherian schemes of finite dimension over
a field of characteristic exponent $p$,
we study the triangulated categories of $\ZZ[1/p]$-linear
mixed motives obtained from $\cdh$-sheaves with transfers.
We prove that these have many of the expected properties.
In particular,
the formalism of the six operations holds in this context.
When we restrict ourselves to regular schemes,
we also prove that these categories of motives
are equivalent to the more classical
triangulated categories of mixed motives
constructed in terms of Nisnevich sheaves with transfers.
Such a program is achieved by comparing these various
triangulated categories of motives with modules over
motivic Eilenberg-MacLane spectra.
\end{abstract}

\maketitle

\setcounter{tocdepth}{3}
\tableofcontents
The main advances of the actual theory of mixed motivic complexes
over a field come from the fact they are defined integrally.
 Indeed, this divides the theory in two variants,
 the Nisnevich one and the \'etale one.
 With rational coefficients, the two theories agree and share their
 good properties. But with integral coefficients,
 their main success comes from the comparison
 of these two variants, the so-called Beilinson-Lichtenbaum
 conjecture which was proved by Voevodsky and gave the solution
 of the Bloch-Kato conjecture.

One of the most recent works in the theory has been devoted
 to extend the definitions in order to get the 6 operations of
 Grothendieck and to check they satisfy the required formalism;
 in chronological order: an unpublished work of Voevodsky, \cite{ayoub},
 \cite{CD3}. While the project has been finally completely realized
 with rational coefficients in \cite{CD3}, the case of integral
 coefficients remains unsolved. In fact, this is half true:
 the \'etale variant is now completely settled:
 see \cite{ayoub5}, \cite{CD4}.

But the Nisnevich variant is less mature.
Markus Spitzweck \cite{spit2} has constructed a motivic ring spectrum
over any Dedekind domain, which allows to define motivic cohomology
of arbitrary schemes, and even triangulated categories
of motives on a general base (under the form of modules over the
pullbacks of the motivic ring spectrum over $\Spec(\ZZ)$).
However, at this moment, there is no proof that Spitzweck's
motivic cohomology satisfies
the absolute purity theorem, and we do not know how to compare
Spitzweck's construction with triangulated categories of motives
constructed in the language of algebraic correspondences (except for fields).
What is concretely at stake is the theory of algebraic cycles:
 we expect that motivic cohomology of a regular scheme 
 in degree $2n$ and twist $n$ agrees with the Chow group
 of $n$-codimensional cycles of $X$.
 Let us recall for example that the localization
 long exact sequence for higher Chow groups and the existence
 of a product of Chow groups of regular schemes are
 still open questions in the arithmetic case
 (\emph{i.e.} for schemes of unequal residual characteristics). For sake of completeness, let us recall that the localization long exact sequence in equal characteristic already is the fruit of non trivial contributions of Spencer Bloch \cite{Bloch1,Bloch2} and Marc Levine \cite{localization}.
Their work involves moving lemmas which 
are generalizations of the classical moving lemma used to understand the intersection product of cycles \cite{Ful}.

 Actually, Suslin and Voevodsky have already provided
 an intersection theoretic basis for the integral definition
 of Nisnevich motivic complexes: the theory of relative
 cycles of \cite[chap. 2]{FSV}. Then, along the lines drawn
 by Voevodsky, and especially the homotopy theoretic setting
 realized by Morel and Voevodsky, it was at least possible
 to give a reasonable definition of such a theory
 over an arbitrary base,
 using Nisnevich sheaves with transfers over this base,
 and the methods of $\AA^1$-homotopy and $\PP^1$-stabilization:
 this was done in \cite[Sec. 7]{CD3}.
 Interestingly enough, the main technical issue of this construction is to
 prove that these motivic complexes satisfy the existence
 of the localization triangle:
 $$j_!\, j^*(M)\to M\to i_*\, i^*(M)\to j_!\, j^*(M)[1]$$
 for any closed immersion $i$ with open complement $j$.
 This echoes much with the question
 of localization sequence for higher Chow groups.

In our unsuccessful efforts to prove this property with integral coefficients,
 we noticed two things: the issue of dealing with singular schemes
 (the property is true for smooth schemes over any base, and,
 with rational coefficients,
 for any closed immersion between excellent geometrically unibranch
 scheme); the fact this property implies $\cdh$-descent
 (i.e. Nisnevich descent together with descent by blow ups).
 Moreover, in \cite{CD4}, we show that, at least for torsion coefficients,
 the localization property for \'etale motivic complexes is true without any restriction,
 but this is due to rigidity properties ({\it \`a la} Suslin)
 which only hold \'etale locally,
 and for torsion coefficients.

 Therefore, the idea of replacing Nisnevich topology by a finer one, 
 which allows to deal with singularities, but remains compatible with algebraic cycles,
 becomes obvious. The natural choice goes to the $\cdh$-topology: in
 Voevodsky's work \cite{FSV}, motivic (co)homology of smooth schemes
 over a field is naturally extended to schemes of finite type by
 $\cdh$-descent in characteristic zero (or, more generally, if we admit
 resolution of singularities), and S.~Kelly's thesis~\cite{kelly} generalizes this
result to arbitrary perfect fields of characteristic $p>0$, at least
with $\ZZ[1/p]$-linear coefficients.

\bigskip

In this work, we prove that if one restricts to noetherian schemes
of finite dimension over
 a prime field (in fact, an arbitrary perfect field) $k$,
 and if we invert solely the characteristic exponent of $k$,
 then mixed motives built out of $\cdh$-sheaves with transfers
 (Definition \ref{df:cdh-motives}) do satisfy 
 the localization property: Theorem \ref{thm:continuityabsDMcdh}.
 Using the work of Ayoub,
 it is then possible to get the complete 6 functors formalism for
 these $\cdh$-motives. Note that we also prove
 that these $\cdh$-motives agree with the Nisnevich ones for
 regular $k$-schemes -- hence proving that the original
 construction done in \cite[Def. 11.1.1]{CD3} is meaningful if
 one restricts to regular schemes of equal characteristic
 and invert the residue characteristic
 (see Corollary \ref{cor:DM_reg&6functors} for a precise account).

The idea is to extend a result of
 R\"ondigs and  {\O}stv{\ae}r, which identifies motivic complexes
 with modules over the motivic Eilenberg-MacLane spectrum
 over a field of characteristic $0$.
 This was recently generalized to perfect fields
 of characteristic $p>0$, up to inverting $p$,
 by Hoyois, Kelly and {\O}stv{\ae}r \cite{HKO}.
 Our main result, proved in Theorem \ref{thm:comparisoncdh},
 is that this property holds for arbitrary noetherian $k$-schemes
 of finite dimension provided we use $\cdh$-motives 
 and invert the exponent characteristic $p$ of $k$ in their coefficients.
For any noetherian $k$-scheme
of finite dimension $X$ with structural map $f:X\to\Spec(k)$, let us put
$H\ZZ_{X/k}=\derL f^*(H\ZZ_k)$. Then there is a canonical equivalence
of triangulated categories
$$H\ZZ_{X/k}[1/p]\Mod\simeq\DM_\cdh(X,\ZZ[1/p])\, .$$
 One of the ingredients is to prove this result for
 Nisnevich motivic complexes with $\ZZ[1/p]$-coefficients
 if one restricts to noetherian regular $k$-schemes 
 of finite dimension: see Theorem \ref{thm:comparisonnis}.
 The other ingredient is to use Gabber's refinement 
 of de~Jong resolution of singularities by alteration
 via results and methods from Kelly's thesis.
 
We finally prove the stability of the notion of constructibility
 for $\cdh$-motives up to inverting the
 characteristic exponent in Theorem \ref{thm:constructmotivic}.
 While the characteristic $0$ case can be obtained
 using results of \cite{ayoub}, the positive characteristic case
 follows from a geometrical argument of Gabber (used in
 his proof of the analogous fact for torsion \'etale sheaves).
 We also prove a duality theorem for schemes of finite type over a field
 (\ref{thm:duality}), and describe cycle cohomology of Friedlander and Voevodsky
using the language of the six functors (\ref{thm:bivcycleDM6op}).
In particular, Bloch's higher Chow groups
and usual Chow groups of schemes of finite type over a field are
are obtained via the expected formulas (see \ref{cor:higherChow} and \ref{cor:Chow}).

\bigskip

We would like to thank Offer Gabber for pointing
out Bourbaki's notion of $n$-gon\-flement, $0\leq n\leq \infty$.
We also want to warmly thank the referee for many precise and constructive
comments and questions, which helped us to greatly improve the readability
of this article.


\section*{Conventions}

We will fix a perfect base field $k$ of characteristic exponent $p$
 -- the case where $k$ is a prime field is enough.
All the schemes appearing in the paper are assumed to be noetherian
 of finite dimension.

We will fix a commutative ring $R$ which will serve as
 our coefficient ring.

\section{Motivic complexes and spectra}

In \cite[chap. 5]{FSV}, Voevodsky introduced the category of motivic complexes
 $\DM_-^{\mathit{eff}}(S)$
 over a perfect field with integral coefficients, a candidate for a conjectural theory
 described by Beilinson. Since then, several generalizations to more general bases have been proposed.

In \cite{CD3}, we have introduced the following ones over a general base
noetherian scheme $S$:
\begin{num} \label{num:shtr} \textit{The Nisnevich variant}.-- 
Let $\Lambda$ be the localization of $\ZZ$ by the prime numbers
 which are invertible in $R$.
The first step is to consider the category $\smcx{\Lambda,S}$ whose ojects
 are smooth separated $S$-schemes of finite type
 and morphisms between $X$ and $Y$
 are \emph{finite $S$-correspondences from $X$ to $Y$ with coefficients
 in $\Lambda$}
 (see \cite[Def. 9.1.8]{CD3} with $\mathscr P$ the category
  of smooth separated morphisms of finite type).\footnote{Recall:
  a finite $S$-correspondence from $X$ to $Y$ with coefficients
	 in $\Lambda$ is
  an algebraic cycle in $X \times_S Y$ with $\Lambda$-coefficients 
  such that:
  \begin{enumerate}
  \item its support is finite equidimensional over $X$,
  \item it is a relative cycles over $X$ in the sense of Suslin and Voevodsky
	 (cf. \cite[chap. 2]{FSV}) -
	 equivalently it is a special cycle over $X$
	 (cf. \cite[def. 8.1.28]{CD3}),
	\item it is $\Lambda$-universal
	 (cf. \cite[def. 8.1.48]{CD3}).
\end{enumerate}
When $X$ is geometrically unibranch,
 condition (2) is always fulfilled (cf. \cite[8.3.26]{CD3}).
When $X$ is regular of the characteristic exponent of
 any residue field of $X$ is invertible in $\Lambda$,
 condition (3) is always fulfilled (cf. \cite[8.3.29]{CD3}
 in the first case). Everything gets much simpler when we work
 locally for the $\cdh$-topology; see \cite[Chap.~2, 4.2]{FSV}.

Recall also for future reference this definition makes sense
 even if $X$ and $Y$ are singular of finite type over $S$.}
Taking the graph of a morphism between smooth $S$-schemes,
 one gets a faithful functor $\gamma$ from the usual category
 of smooth $S$-schemes to the category $\smcx{\Lambda,S}$.
 
Then one defines the category $\shtr_\nis(S,R)$ of \emph{sheaves with transfers over $S$}
 as the category of presheaves $F$ of $R$-modules over $\smcx{\Lambda,S}$
 whose restriction to the category of smooth $S$-schemes $F \circ \gamma$ is
 a sheaf for the Nisnevich topology. Essentially according to the
 original proof of Voevodsky over a field (see \cite[10.3.3 and 10.3.17]{CD3}
 for details), this is a symmetric monoidal
 Grothendieck abelian category.

The category $\DM(S,R)$ of Nisnevich motivic spectra over $S$ is defined
 by applying the process of $\AA^1$-localization, and then of $\PP^1$-stabilization,
 to the (adequate model category structure corresponding to the) derived category of $\shtr_\nis(S,R)$; see
 \cite[Def. 11.1.1]{CD3}.
By construction, any smooth $S$-scheme $X$ defines
 a (homological) motive $M_S(X)$ in $\DM(S,R)$
  which is a compact object. Moreover, the triangulated category
  $\DM(S,R)$ is generated by Tate twists of such
  homological motives, i.e. by objects of the
  form $M_S(X)(n)$ for a smooth $S$-scheme $X$,
  and an integer $n \in \ZZ$.
\end{num}

\begin{rem}
When $S=\Spec(K)$ is the spectrum of a perfect field,
 the triangulated category $\DM(S,\ZZ)$ contains as a full and faithful subcategory
 the category $\DM_-^{\mathit{eff}}(K)$ defined in \cite[chap. 5]{FSV}.
 This follows from the description of $\AA^1$-local objects in this case
 and from the cancellation theorem of Voevodsky \cite{cancel}
 (see for example \cite[Sec. 4]{Deg9}
 for more details).
\end{rem}

\begin{num}\label{num:generalized_DM}
\textit{The generalized variants}.-- This variant
 is an \emph{enlargement}\footnote{See \cite[1.4.13]{CD3}
 for a general definition of this term.}
 of the previous context.
 However, at the same time, one can consider several possible Grothendieck topologies $t$:
 the Nisnevich topology $t=\nis$, the $\cdh$-topology $t=\cdh$,
 the \'etale topology $t=\et$, or the $\h$-topology $t=\h$.
 
Instead of using the category $\smcx{\Lambda,S}$,
 we consider the larger category $\sftcx{\Lambda,S}$
 made by all separated $S$-schemes of
 finite type whose morphisms are made by the finite $S$-correspondences
 with coefficients in $\Lambda$
 as in the previous paragraph (see again \cite[9.1.8]{CD3} with
 $\mathscr P$ the class of all separated morphisms of finite type).

Then we can still define the category $\ushtr_t(S,R)$ of generalized $t$-sheaves
 with transfers over $S$
 as the category of additive presheaves of $R$-modules over \smash{$\sftcx{\Lambda,S}$} 
 whose restriction to $\sft_S$ is a sheaf for the $\cdh$ topology.
 This is again a well suited Grothendieck abelian category
 (by which we mean that, using the terminology of \cite{CD3},
 when we let $S$ vary, we get an abelian premotivic category which is compatible with the topology $t$; see \cite[Sec. 10.4]{CD3}).
Moreover we have natural adjunctions:
\begin{equation} \label{eq:nis2cdh_abelian}
\xymatrix@=30pt{
\shtr_\nis(S,R)\ar@<2pt>^{\rho_!}[r]
 & \ushtr_\nis(S,R)\ar@<2pt>^{a_\cdh^*}[r]\ar@<2pt>^{\rho^*}[l]
 & \ushtr_\cdh(S,R)\ar@<2pt>[l]
}
\end{equation}
where $\rho^*$ is the natural restriction functor
 and $a_\cdh^*$ is the associated $\cdh$-sheaf with
 transfers functor (see \emph{loc. cit.})

Finally,
 one defines the category $\uDM_t(S,R)$ of generalized motivic $t$-spectra over $S$
 and coefficients in $R$
 as the triangulated category obtained by $\PP^1$-stabilization and $\AA^1$-localization of the (adequate model category structure corresponding to the) derived category of $\ushtr_t(S,R)$.
 
Note that in the generalized context, any $S$-scheme $X$ defines
 a (homological) $t$-motive $M_S(X)$ in $\uDM_t(S,R)$
  which is a compact object and depends covariantly on $X$. 
  This can even be extended to
   simplicial $S$-schemes (although we might then obtain
   non compact objects). Again, the triangulated category
  $\uDM_t(S,R)$ is generated by objects of the form
  $M_S(X)(n)$ for a smooth $S$-scheme $X$ and an integer $n \in \ZZ$.

Thus, we have three variants of motivic spectra.
Using the adjunctions \eqref{eq:nis2cdh_abelian}
 (which are Quillen adjunctions for suitable underlying model categories),
 one deduces adjunctions made by exact functors as follows:
\begin{equation} \label{eq:nis2cdh_tri}
\xymatrix@=30pt{
\DM(S,R)\ar@<2pt>^{\derL \rho_!}[r]
 & \uDM(S,R)\ar@<2pt>^{\derL a_\cdh^*}[r]\ar@<2pt>^{\derR \rho^*}[l]
 & \uDM_\cdh(S,R)\ar@<2pt>[l]
}
\end{equation}
The following assertions are consequences of the model category structures
 used to get these derived functors:
\begin{enumerate}
\item for any smooth $S$-scheme $X$ and any integer $n \in \ZZ$,
 $\derL \rho_!\big(M_S(X)(n)\big)=\uM_S(X)(n)$.
\item for any $S$-scheme $X$ and any integer $n \in \ZZ$,
 $\derL a_\cdh ^*\big(\uM_S(X)(n)\big)=\uM_S(X)(n)$.
\end{enumerate}
 \end{num}

The following proposition is a formal consequence of
 these definitions:
\begin{prop}
The category $\uDM_\cdh(S,R)$ is the localization
 of $\uDM(S,R)$ obtained by inverting
 the class of morphisms of the form:
$$
\uM_S(X_\bullet) \xrightarrow{p_*} \uM_S(X)
$$
for any $\cdh$-hypercover $p$ of any $S$-scheme $X$.
Moreover, the functor $a_\cdh$ is the canonical projection functor.
\end{prop}

The definition that will prove most useful is the following one.
\begin{df}\label{df:cdh-motives}
Let $S$ be any noetherian scheme.

One defines the triangulated category $\DM_\cdh(S,R)$
 of $\cdh$-motivic spectra, as the full localizing triangulated subcategory
 of $\uDM_{\cdh}(S,R)$ generated by motives of the form
 $\uM_S(X)(n)$ for a smooth $S$-scheme $X$ and an integer $n \in \ZZ$.
\end{df}

\begin{num}
 These categories for various base schemes $S$ are equipped with a basic
 functoriality
 ($f^*$, $f_*$, $f_\sharp$ for $f$ smooth, $\otimes$ and $\uHom$)
 satisfying basic properties. In  \cite{CD3},
 we have summarized these properties
 saying that $\DM(-,R)$ is a \emph{premotivic triangulated category} --
 see 1.4.2 for the definition and 11.1.1 for the construction.
\end{num}

\section{Modules over motivic Eilenberg-MacLane spectra}


\subsection{Symmetric Tate spectra and continuity}

\begin{num}
Given a scheme $X$ we write $\Sp_X$ for the category of
symmetric $T$-spectra, where $T$ denotes a cofibrant resolution
of the projective line $\mathbf{P}^1$ over $X$
(with the point at infinity as a base point, say)
in the projective model structure
 of pointed Nisnevich simplicial sheaves of sets.
We will consider $\Sp_X$ as combinatorial
stable symmetric monoidal model category,
obtained as the $T$-stabilization
of the $\AA^1$-localization of the projective model category structure
on the category of pointed Nisnevich simplicial sheaves of sets
on the site $\sm_X$ of smooth separated $X$-schemes of finite type.
The corresponding homotopy category
$$\mathrm{Ho}(\Sp_X)=\SH(X)$$
is thus the stable homotopy category of schemes over $X$,
as considered by Morel, Voevodsky and various other authors.
This defines a motivic triangulated category in the sense
of \cite{CD3}: in other words, thanks to Ayoub's thesis \cite{ayoub,ayoub2},
we have the whole formalism of the six operations in $\SH$.
We note that the categories $\SH(X)$ can be defined as the
homotopy categories of their $(\infty,1)$-categorical counterparts;
see \cite[2.3]{robalo} and \cite[Appendix C]{hoyois}.

%
\end{num}

\begin{num}
In \cite{CD3}, we have introduced the notion
 of continuity for a premotivic category $\T$
 which comes from the a premotivic model category.
 In the sequel, we will need to work in a more 
 slightly general context,
 in which we do not consider a monoidal structure.
 Therefore,
  we will recast the definition of continuity 
	for complete triangulated $\sm$-fibred categories
	over $\sch$
	(see \cite[1.1.12, 1.3.13]{CD3} for the definitions; in particular, the adjective `complete' refers to the existence of right adjoints for the pullback functors).

Here $\sch$ will be a full subcategory of the category
 of schemes stable by smooth base change
 and $\mathcal F$ will be a class of affine morphisms
 in $\sch$.\footnote{The examples we will use
 here are: $\sch$ is the category of regular (excellent) $k$-schemes
 or the category of all noetherian finite dimensional (excellent) $k$-schemes;
 $\mathcal F$ is the category of dominant affine morphisms or the category
 of all affine morphisms.}
\begin{df}
Let $\T$ be a complete triangulated $\sm$-fibred category
	over $\sch$ and $c$ be a small family
 of cartesian sections $(c_i)_{i\in I}$ of $\T$. 

We will say that $\T$ is \emph{$c$-generated}
if, for any scheme $X$ in $\sch$,
 the family of objects $c_{i,X}$, $i\in I$,
 form a generating family of the triangulated category.
We will then define $\T_c(X)$ as the smallest thick subcategory
of $\T(X)$ which contains the elements of of the form
$f_\sharp f^*(c_{i,X})=f_\sharp(c_{i,Y})$, for any separated
smooth morphism $f:Y\to X$ and any $i\in I$.
The objects of $\T_c(X)$ will be called \emph{$c$-constructible}
(or simply \emph{constructible}, when $c$ is clearly determined
by the context).
\end{df}

\begin{rem}\label{rem:constructcompact}
If for any $i \in I$,
 the objects $c_{i,X}$ are compact,
 then $\T_c(X)$ is the category of compact objects of $\T(X)$
 and so does not depend on $c$.

When $\T$ has a symmetric monoidal structure,
 or in other words, is a premotivic category,
 and if we ask that $c$ is stable by tensor product,
 then $c$ is what we call a \emph{set of twists} in \cite[1.1.d]{CD3}.
This is what happens in practice (e.g. for $\T=\SH$, $\DM$ or $\DM_\cdh$),
and the family $c$ consists of the Tate twist $\un_X(n)$
 of the unit object for $n \in \ZZ$.
Moreover, constructible objects coincide with compact objects for
 $\SH$, $\DM$ and $\DM_\cdh$.
\end{rem}
%

For short, a $(\sch,\mathcal F)$-pro-scheme will
 be a pro-scheme $(S_\alpha)_{\alpha \in A}$ with values
 in $\sch$, whose transition morphisms are in $\mathcal F$,
 which admits a projective limit $S$ in the category
 of schemes such that $S$ belongs to $\sch$.
The following definition is a slightly more general
version of \cite[4.3.2]{CD3}.
\end{num}

\begin{df}\label{df:continuous}
Let $\T$ be a $c$-generated complete triangulated
 $\sm$-fibred category over $\sch$.

We say that $\T$ is \emph{continuous with respect to $\mathcal F$},
 if given any $(\sch,\mathcal F)$-pro-scheme $(X_\alpha)$ with limit $S$,
 for any index $\alpha_0$, any object $E_{\alpha_0}$ in $\T(X_{\alpha_0})$,
 and any $i \in I$, the canonical map
$$
\varinjlim_{\alpha\geq \alpha_0}
 \Hom_{\T(X_\alpha)}(c_{i,X_\alpha},E_\alpha)
 \rightarrow \Hom_{\T(X)}(c_{i,S},E),
$$
is bijective, where $E_\alpha$ is the pullback of $E_{\alpha_0}$
 along the transition morphism $X_\alpha \rightarrow X_ {\alpha_0}$,
 while $E$ is the pullback of $E_{\alpha_0}$ along the projection
 $X\to X_{\alpha_0}$
\end{df}

\begin{ex} \label{ex:continuity}
\begin{enumerate}
\item The premotivic category
 $\SH$ on the category of noetherian finite dimensional
 schemes satisfies continuity without restriction (\emph{i.e.}
 $\mathcal F$ is the category of all affine morphisms).
 This is a formal consequence of \cite[Proposition C.12]{hoyois}
 and of \cite[Lemma 6.3.3.6]{htt}, for instance.
\item According to \cite[11.1.4]{CD3},
 the premotivic triangulated categories
 $\DM$ and $\DM_\cdh$, defined over the category of all schemes,
 are continuous with respect to dominant affine morphisms.
 (Actually, this example is the only reason why we
 introduce a restriction on the transition morphisms
 in the previous continuity property.)
\end{enumerate}
\end{ex}

%
The following proposition
is a little variation on \cite[4.3.4]{CD3},
in the present slightly generalized context:
\begin{prop}\label{prop:continuity&2lim}
Let $\T$ be a $c$-generated complete triangulated
 $\sm$-fibred category over $\sch$
 which is continuous with respect to $\mathcal F$.
Let $(X_\alpha)$ be a $(\sch,\mathcal F)$-pro-scheme 
 with projective limit $X$ and let $f_\alpha:X \rightarrow X_\alpha$
 be the canonical projection.
 
 For any index $\alpha_0$ and any objects $M_{\alpha_0}$ and $E_{\alpha_0}$
 in $\T(S_{\alpha_0})$, if $M_{\alpha_0}$ is $c$-constructible,
 then the canonical map
$$
\varinjlim_{\alpha\geq \alpha_0}
 \Hom_{\T(S_\alpha)}(M_\alpha,E_\alpha)
 \rightarrow \Hom_{\T(S)}(M,E),
$$
is bijective, where $M_\alpha$ and $E_\alpha$
are the respective pullbacks of $M_{\alpha_0}$ and $E_{\alpha_0}$
 along the transition morphisms $S_\alpha \rightarrow S_ {\alpha_0}$,
 while $M=f^*_{\alpha_0}(M_{\alpha_0})$ and $E=f^*_{\alpha_0}(E_{\alpha_0})$.
 
 Moreover, the canonical functor:
$$
2\mbox{-}\ilim_{\alpha} \T_c(X_\alpha)
 \xrightarrow{2\mbox{-}\ilim_{\alpha}(f_\alpha^*)}
  \T_c(X)
$$
is an equivalence of triangulated categories.
\end{prop}
The proof is identical to that of \emph{loc. cit.}

\begin{prop}\label{prop:exactness_regular}
Let $f:X \rightarrow Y$ be a regular morphism of schemes.
Then the pullback functor 
$$
f^*:\Sp_Y \rightarrow \Sp_X
$$
of the premotivic model category of Tate spectra
 (relative to simplicial sheaves)
 preserves stable weak $\AA^1$-equivalences
 as well as $\AA^1$-local fibrant objects.
\end{prop}
\begin{proof}
This property is local in $X$ so that
 replacing $X$ (resp. $Y$) by a suitable affine
 open neighbourhood of any point $x \in X$ (resp. $f(x)$),
 we can assume that $X$ and $Y$ are affine.

Then, according to Popescu's theorem (as stated in Spivakovsky's article \cite[Th. 1.1]{spivak}), the morphism $f$
 can be written as a projective limit of smooth morphisms
 $f_\alpha:X_\alpha \rightarrow Y$.
 By a continuity argument (in the context of sheaves of sets!),
 as each functor $f^*_\alpha$ commutes with small limits and colimits,
 we see that the functor $f^*$ commutes with small colimits as well as with
 finite limits. These exactness properties imply that the functor $f^*$ preserves stalkwise simplicial weak equivalences.
 One can also check that, for any Nisnevich
 sheaves $E$ and $F$ on $\sm_Y$, the canonical map
\begin{equation}\label{eq:prop:exactness_regular}
 f^*\uHom(E,F)\to\uHom(f^*(E),f^*(F))
\end{equation}
 is an isomorphism (where $\uHom$ denotes the internal Hom of the category of sheaves), at least when $E$ is a finite colimit of representable sheaves.
Since the functor $f^*$ preserves projections of the form $\AA^1\times U\to U$, this readily implies that,
 if $L$ denotes the explicit
 $\AA^1$-local fibrant replacement functor defined in \cite[Lemma 3.21, page 93]{MV},
 then, for any simplicial
 sheaf $E$ on $\sm_Y$, the map $f^*(E)\to f^*(L(E))$ is an $\AA^1$-equivalence with fibrant
 $\AA^1$-local codomain. Therefore, the functor $f^*$ preserves both
 $\AA^1$-equivalences and $\AA^1$-local fibrant objects at the level of
 simplicial sheaves. Using the isomorphism \eqref{eq:prop:exactness_regular},
 it is easy to see that $f^*$ preserves $\AA^1$-local motivic
 $\Omega$-spectra. Given that one can turn a levelwise $\AA^1$-local fibrant
 Tate spectrum into a motivic $\Omega$-spectrum by a suitable filtered
 colimit of iterated $T$-loop space functors, we see that there exists
 a fibrant replacement functor $R$ in $\Sp_Y$ such that, for any
 Tate spectrum $E$ over $Y$, the map $f^*(E)\to f^*(R(E))$ is a stable
 $\AA^1$-equivalence with fibrant codomain. This implies that $f^*$
 preserves stable $\AA^1$-equivalences.
\end{proof}

\begin{cor}\label{cor:cartsectionReg}
Let $A$ be a commutative monoid in $\Sp_k$.
Given a regular $k$-scheme $X$ with structural map $f:X\to\Spec(k)$,
 let us put $A_X=f^*(R)$.
Then, for any $k$-morphism between regular $k$-schemes $\varphi:X\to Y$,
the induced map $\derL\varphi^*(A_Y)\to A_X$ is an isomorphism in $\SH(X)$.
\end{cor}

\begin{proof}
It is clearly sufficient to prove this property when $Y=\Spec(k)$, in which
case this is a direct consequence of the preceding proposition.
\end{proof}

We will use repeatedly the following easy
 fact to get the continuity property.
 
\begin{lm} \label{lm:trivial_continuity}
Let
$$\varphi^*:\T \leftrightarrows \T':\varphi_*$$
be an adjunction of complete triangulated $\sm$-fibred
 categories.
We make the following assumptions:
\begin{itemize}
\item[(i)] There is a small family $c$ of cartesian sections of $\T$
such that $\T$ is $c$-generated.
\item[(ii)] The functor $\varphi_*$ is conservative
(or equivalently, $\T'$ is $\varphi^*(c)$-generated;
by abuse, we will then write $\varphi^*(c)=c$ and
will say that $\T'$ is $c$-generated).
\item[(iii)] The functor $\varphi_*$ commutes with the operation $f^*$
 for any morphism $f \in \mathcal F$.
\end{itemize}
 Then, if $\T$ is continuous with respect to $\mathcal F$,
 the same is true for $\T'$.
\end{lm}
\begin{proof}
Let $c=(c_{i,?})_{i \in I}$.
 For any morphism $f:Y \rightarrow X$
 in $\mathcal F$,
 any object $E \in \T'(X)$ 
 and any $i \in I$, one has a canonical isomorphism:
\begin{align*}
\Hom_{\T'(Y)}(c_{i,Y},f^*(E))
&=\Hom_{\T'(Y)}(\varphi^*(c_{i,Y}),f^*(E))
\simeq \Hom_{\T(Y)}(c_{i,Y},\varphi_*f^*(E)) \\
&\simeq \Hom_{\T(Y)}(c_{i,Y},f^*\varphi_*(E))\, .
\end{align*}
This readily implies the lemma.
\end{proof}

\begin{ex}\label{ex:modules_reg}
Let $\reg$ be the category of regular $k$-schemes
 with morphisms all morphisms of $k$-schemes.

Let $(A_X)_{X \in \reg}$ be a cartesian section
 of the category of commutative monoids in the category
 of Tate spectra (\emph{i.e.} a strict commutative ring spectrum
 stable by pullbacks with respect to morphisms in $\reg$).
In this case, we have defined in \cite[7.2.11]{CD3}
 a premotivic model category over $\reg$
 whose fiber $A_X\Mod$ over a scheme $X$ in $\reg$ is the homotopy category of the symmetric
 monoidal stable model category of
 $A_X$-modules\footnote{In order to apply
 this kind of construction, we need to know that the
 model category of symmetric Tate spectra in simplicial
 sheaves satisfies the monoid axiom of Schwede and Shipley \cite{SS}. This is proved explicitely
 in \cite[Lemma 4.2]{hoyois2}, for instance.}
 (i.e. of Tate spectra over $S$, equiped with an action
 of the commutative monoid $A_X$).
 Since Corollary \ref{cor:cartsectionReg}
 ensures that $(A_X)_{X \in \reg}$ is a homotopy cartesian
 section in the sense of \cite[7.2.12]{CD3},
 according to \cite[7.2.13]{CD3}, there exists a
 premotivic adjunction:
$$
L_A:\SH \leftrightarrows A\Mod:\mathcal O_A
$$
of triangulated premotivic categories over $\reg$,
 such that $L_A(E)=A_S \wedge E$
 for any spectrum $E$ over a scheme $S$ in $\reg$.
 Lemma \ref{lm:trivial_continuity}
 ensures that $A\Mod$ is continuous
 with respect to affine morphisms in $\reg$.
\end{ex}

\subsection{Motivic Eilenberg-MacLane spectra over regular $k$-schemes}

\begin{num}\label{num:HZ-mod_reg}
There is a canonical premotivic adjunction:
\begin{equation} \label{eq:SH<->DM}
\varphi^*:\SH \leftrightarrows \DM:\varphi_*
\end{equation}
(see \cite[11.2.16]{CD3}).
It comes from an adjunction of the premotivic model
 categories of Tate spectra built out of simplicial sheaves of sets
 and of complexes of sheaves with transfers respectively (see \ref{num:shtr}):
\begin{equation} \label{eq:SH<->DM_model}
\tilde \varphi^*:\Sp \leftrightarrows \Sp^{tr}:\tilde \varphi_*.
\end{equation}
In other words, we have $\varphi^*=\derL\tilde\varphi^*$ and
$\varphi_*=\derR\tilde\varphi_*$ (strictly speaking,
we can construct the functors $\derL\tilde\varphi^*$ and $\derR\tilde\varphi_*$ so that these equalities are true at the level of objects).
Recall in particular from \cite[10.2.16]{CD3} 
 that the functor $\tilde\varphi_*$
 is composed by the functor $\tilde \gamma_*$ with values in Tate spectra
 of Nisnevich sheaves of $R$-modules (without transfers),
 which \emph{forgets transfers} and by the functor induced
 by the right adjoint of the Dold-Kan equivalence.
We define, for any scheme $X$:
\begin{equation}\label{eq:defHR}
\HZ_X=\tilde\varphi_*(R_X)\, .
\end{equation}
This is Voevodsky's motivic Eilenberg-MacLane spectrum over $X$, originally defined in \cite[6.1]{V3}. In the case where
$X=\Spec(K)$ for some commutative ring $K$, we sometimes
write
\begin{equation}\label{eq:defHRaffine}
\HZ_K=\HZ_{\Spec K}\, .
\end{equation}

According to \cite[6.3.9]{CD3}, the functor $\tilde \gamma_*$ preserves
 (and detects) stable $\AA^1$-equi\-va\-lences. We deduce that
 the same fact is true for $\tilde \varphi_*$. Therefore, we have
 a canonical isomorphism
 $$\HZ_X\simeq\varphi_*(R_X)\simeq\derR\tilde\varphi_*(R_X)\, .$$
The Tate spectrum $\HZ_X$ is a commutative motivic ring spectrum
 in the strict sense (\emph{i.e.} a commutative monoid in the
 category $\Sp_X$). We denote by $\HZ_X\Mod$
 the homotopy category of $\HZ_X$-modules. This defines a
 fibred triangulated category over the category of schemes;
 see \cite[Prop. 7.2.11]{CD3}.
 
 The functor $\tilde \varphi_*$ being weakly monoidal,
 we get a natural structure of a commutative monoid on $\tilde \varphi_*(M)$
for any symmetric Tate spectrum with transfers $M$.
This means that the Quillen adjunction \eqref{eq:SH<->DM_model}
induces a Quillen adjunction from the fibred model
category of $\HZ$-modules to
the premotivic model category of symmetric Tate spectra with transfers\footnote{The fact that the induced adjunction is a Quillen adjunction is obvious: this readily comes from the fact that the forgetful functor from $\HZ$-modules
to symmetric Tate spectra preserves and detects weak equivalences as well as fibrations (by definition).},
and thus defines an adjunction
\begin{equation} \label{eq:Mod<->DM}
t^*:\HZ_X\Mod \leftrightarrows \DM(X,R):t_*
\end{equation}
for any scheme $X$. For any object $E$ of $\SH(X)$, there is a
canonical isomorphism $t^*(\HZ_X\otimes^\derL E)=\varphi^*(E)$.
For any object $M$ of $\DM(X,R)$, when we forget the
$\HZ_X$-module structure on $t_*(M)$,
we simply obtain $\varphi_*(M)$.

Let $f:X \rightarrow S$ be a regular morphism of schemes.
 Then according to Proposition \ref{prop:exactness_regular},
 $f^*=\derL f^*$. In particular, the isomorphism $\tau_f$
 of $\SH(X)$ can be lifted as a morphism of strict ring spectra:
\begin{equation} \label{eq:fibred_struct_strict_HZ}
\tilde \tau_f:f^*(\HZ_S) \rightarrow \HZ_X.
\end{equation}

Let $\reg$ be the category of regular $k$-schemes
 as in Example \ref{ex:modules_reg}. 
\end{num}

\begin{prop}\label{prop:adjHZModDM}
The adjunctions \eqref{eq:Mod<->DM} define a premotivic adjunction
$$t^*:\HZ\Mod\rightleftarrows\DM(-,R):t_*$$
over the category $\reg$ of regular $k$-schemes.
\end{prop}

\begin{proof}
We already know that this is a an adjunction of fibred
categories over $\reg$ and that $t^*$ is (strongly)
symmetric monoidal. Therefore, it is sufficient to check that
$t^*$ commutes with the operations $f_\sharp$ for any smooth
morphism between regular $k$-scheme $f:X\to S$ (via the canonical exchange
map). For this, it is sufficient to check what happens on
free $\HZ_X$-modules (because we are comparing exact functors
which preserve small sums, and because the smallest localizing subcategory of
$\HZ_X\Mod$ containing free $\HZ_X$-modules is $\HZ_X\Mod$).
For any object $E$ of $\SH(X)$, we have, by the projection
formula in $\SH$, a canonical isomorphism in $HZ_S\Mod$:
$$\derL f_\sharp(\HZ_X\otimes^\derL E)\simeq\HZ_S\otimes^\derL\derL f_\sharp(E)\, .$$
Therefore, formula $t^*(\HZ_X\otimes^\derL E)=\varphi^*(E)$
tells us that $t^*$ commutes with $f_\sharp$ when restricted to
free $\HZ_X$-modules, as required.
\end{proof}
%

\section{Comparison theorem: regular case}

The aim of this section is to prove the following result:

\begin{thm}\label{thm:comparisonnis}
Let $R$ be a ring in which the characteristic exponent of $k$
is invertible. Then the premotivic adjunction of
Proposition \ref{prop:adjHZModDM}
is an equivalence of premotivic categories over $\reg$.
In particular, for any regular noetherian scheme of finite dimension $X$ over $k$,
we have a canonical equivalence of symmetric monoidal triangulated
categories
$$\HZ_X\Mod\simeq\DM(X,R)\, .$$
\end{thm}

The preceding theorem tells us that 
 the 6 operations constructed on $\DM(-,R)$ 
 in \cite[11.4.5]{CD3}, behave appropriately
 if one restricts to regular noetherian schemes
 of finite dimension over $k$:
\begin{cor}\label{cor:DM_reg&6functors}
Consider the notations of paragraph \ref{num:HZ-mod_reg}.
\begin{enumerate}
\item The functors $\varphi^*$ and $\varphi_*$
 commute with the operations $f^*$, $f_*$ (resp. $p_!$, $p^!$)
 for any morphism $f$ (resp. separated morphism
 of finite type $p$) between regular $k$-schemes.
\item The premotivic category $\DM(-,R)$ over $\reg$ satisfies:
\begin{itemize}
\item the localization property;
\item the base change formula
 ($g^*f_!\simeq f'_!g^{\prime *}$,
 with notations of \cite[11.4.5, (4)]{CD3});
\item the projection formula
 ($f_!(M\otimes f^*(N)) \simeq f_!(M) \otimes N$, 
 with notations of \cite[11.4.5, (5)]{CD3}).
\end{itemize}
\end{enumerate}
\end{cor}
\begin{proof}
Point (1) follows from the fact the premotivic adjunction 
 $(L_{\HZ},\mathcal O_{\HZ})$
 satisfies the properties stated for $(\varphi^*,\varphi_*)$
 and that they are true for $(t^*,t_*)$ because it is 
 an equivalence of premotivic categories, due to
 Theorem \ref{thm:comparisonnis}.
The first statement of Point (2) follows from the fact
that the localization property over $\reg$
holds in $\HZ\Mod$, and from the equivalence $\HZ\Mod\simeq\DM(-,R)$
over $\reg$; the remaining two statements follow from Point (2) and
 the fact they are true for $\SH$ (see \cite{ayoub} in the
 quasi-projective case and \cite[2.4.50]{CD3} in the general
 case).
\end{proof}

The proof of Theorem \ref{thm:comparisonnis}
will be given in Section \ref{subsec:proofregcase} (page \pageref{subsec:proofregcase}),
after a few preparations. But before that, we will explain some of its
consequences.

\begin{num}\label{num:basic_motivic_EM_ring}
Let $f:X \rightarrow S$ be a morphism of schemes.
Since \eqref{eq:SH<->DM} is an adjunction of fibred
categories over the category of schemes,
we have a canonical exchange transformation
 (see \cite[1.2.5]{CD3}):
\begin{equation} \label{eq:ex_phi_f_*}
Ex(f^*,\varphi_*):\derL f^*\varphi_* \rightarrow \varphi_*\derL f^*.
\end{equation}
Evaluating this natural transformation on the object $\un_S$
 gives us a map:
$$
\tau_f:\derL f^*(\HZ_S) \rightarrow \HZ_X.
$$
Voevodsky conjectured in \cite{V_OpenPB} the following property:
\begin{conj}[Voevodsky] The map $\tau_f$ is an isomorphism.
\end{conj}
When $f$ is smooth, the conjecture is obviously true
 as $Ex(f^*,\varphi_*)$ is an isomorphism.
\end{num}

\begin{rem}
The preceding conjecture of Voevodsky is closely related
 to the localization property for $\DM$.
In fact,  let us also mention the following result
 which was implicit in \cite{CD3}
 -- as it will not be used in the sequel
  we leave the proof as an exercise for the
	reader.\footnote{Hint: use the fact that $\varphi_*$ commutes
	with $j_\sharp$
	(\cite[6.3.11]{CD3} and \cite[11.4.1]{CD3}).}
\begin{prop}\label{prop:localization&Voevodsky_conj}
We use the notations of Par. \ref{num:basic_motivic_EM_ring}.
Let $i:Z \rightarrow S$ be a closed immersion.
 Then the following properties are equivalent:
\begin{enumerate}
\item[(i)] The premotivic triangulated category $\DM$
 satisfies the localization property with respect to $i$
 (see \cite[2.3.2]{CD3}).
\item[(ii)] The natural transformation $Ex(i^*,\varphi_*)$ is an isomorphism.
\end{enumerate}
\end{prop}
\end{rem}

From the case of smooth morphisms, we get the following particular
 case of the preceding conjecture.

\begin{cor} \label{cor:DM->SH_commutes_reg_pullbacks}
The conjecture of Voevodsky holds
 for any morphism $f:X \rightarrow S$ of regular $k$-schemes.
\end{cor}
\begin{proof}
By transitivity of pullbacks,
 it is sufficient to consider the case where $f=p$ is the structural
 morphism of the $k$-scheme $S$, with $k$ a prime field (in particular,
with $k$ perfect). Since $\DM$ is continuous with respect to projective
systems of regular $k$-schemes with affine transition maps
(because this is the case for $\HZ$-modules, using Theorem \ref{thm:comparisonnis}),
we are reduced to the case where $S$ is smooth over $k$, which is trivial.
\end{proof}

\begin{rem}
The previous result is known to have interesting consequences
 for the motivic Eilenberg-MacLane spectrum $\HZ_X$
 where $X$ is an arbitrary noetherian regular $k$-scheme
 of finite dimension.

For example, we get the following extension of a result
 of Hoyois on a theorem first stated by Hopkins and
 Morel (for $p=1$).
 Given a scheme $X$ as above, the canonical map
$$
\mathit{MGL}_X/(a_1,a_2,\hdots)[1/p] \rightarrow \mathit{H}\ZZ_X[1/p]
$$
from the algebraic cobordism ring spectrum modulo
 generators of the Lazard ring is an isomorphism up to inverting
 the characteristic exponent of $k$.
 This was proved in \cite{hoyois2}, for the base field $k$,
 or, more generally, for any essentially smooth $k$-scheme $X$.

This shows in particular that $\mathit{H}\ZZ_X[1/p]$
 is the universal oriented ring $\ZZ[1/p]$-linear spectrum over $X$ with additive formal group law.
 
 All this story remains true for arbitrary noetherian $k$-schemes of finite dimension if we are eager to replace $\mathit{H}\ZZ_X$ by its $\cdh$-local version:
 this is one of the meanings of Theorem \ref{thm:comparisoncdh} below. Note that, since Spitweck's version of the motivic spectrum has the same relation with algebraic cobordism (see \cite[Theorem 11.3]{spit2}), it coincides with the $\cdh$-local version of $\mathit{H}\ZZ_X$ as well, at least up to $p$-torsion.
\end{rem}

\begin{df}\label{df:HZrelatif}
Let $X$ be a regular $k$-scheme with structural map
$f:X\to\Spec(k)$. We define the \emph{relative motivic
Eilenberg-MacLane spectrum} of $X/k$ by the formula
$$\HZ_{X/k}=f^*(\HZ_{\Spec(k)})$$
(where $f^*:\Sp_k\to\Sp_X$ is the pullback functor
at the level of the model categories).
\end{df} 

\begin{rem}\label{rem:df:HZrelatif}
By virtue of Propositions \ref{prop:exactness_regular}
and Corollary \ref{cor:DM->SH_commutes_reg_pullbacks}, we have
canonical isomorphisms
$$\derL f^*(\HZ_{\Spec{(k)}})\simeq\HZ_{X/k}\simeq\HZ_X\, .$$
Note that, the functor $f^*$ being symmetric monoidal, each
relative motivic Eilenberg-MacLane spectrum $\HZ_{X/k}$
is a commutative monoid in $\Sp_X$.
This has the following consequences.
\end{rem}

\begin{prop}\label{prop:HZmodpremotivicReg}\label{prop:HZmodrelatif1}
For any regular $k$-scheme $X$, there is a canonical equivalence of
symmetric monoidal triangulated categories
$$\HZ_{X/k}\Mod\simeq\HZ_X\Mod\, .$$
In particular, the assignment $X\mapsto\HZ_X\Mod$ defines
a premotivic symmetric monoidal triangulated category $\HZ\Mod$
over $\reg$, which is continuous with respect to any projective system
of regular $k$-schemes with affine transition maps.

Moreover the forgetful functor
$$\HZ\Mod\to\SH$$
commutes with $\derL f^*$ for any $k$-morphism $f:X\to Y$
between regular schemes, and with $\derL f_\sharp$ for any
smooth morphism of finite type between regular schemes.
\end{prop}

\begin{proof}
Since the canonical morphism of commutative
monoids $\HZ_{X/k}\to\HZ_X$ is a stable $\AA^1$-equivalence
the first assertion is a direct consequence of
\cite[Prop. 7.2.13]{CD3}.
The property of continuity is a particular
case of Example \ref{ex:modules_reg}, with $R_X=\HZ_{X/k}$.
For the last part of the proposition,
by virtue of the last assertion of \cite[Prop. 7.1.11 and 7.2.12]{CD3}
we may replace (coherently) $\HZ_X$ by a cofibrant monoid $R_X$ (in the model category
of monoids in $\Sp_X$), in order to apply \cite[Prop. 7.2.14]{CD3}:
The forgetful functor from $R_X$-modules to $\Sp_X$ is a left Quillen
functor which preserves weak equivalences and commutes with $f^*$
for any map $f$ in $\reg$: therefore, this relation remains true after we pass to the
total left derived functors. The case of $\derL f_\sharp$ is similar.
\end{proof}

We now come back to the aim of proving Theorem \ref{thm:comparisonnis}.

\subsection{Some consequences of continuity}

\begin{lm}\label{lm:basechangebycontinuity}
Consider the cartesian square of schemes below.
$$
\xymatrix@=10pt{
X'\ar^q[r]\ar_g[d] & X\ar^f[d] \\
Y'\ar ^p[r] & Y
}
$$
We assume that $Y'$ is the projective limit of
a projective system of $Y$-schemes $(Y_\alpha)$ with
affine flat transition maps, and make the
following assumption.
For any index $\alpha$, if $p_\alpha:Y_\alpha\to Y$
denotes the structural morphism, the base change morphism
associated to the pullback square
$$
\xymatrix@=10pt{
X_\alpha\ar^{q_\alpha}[r]\ar_{g_\alpha}[d] & X\ar^f[d] \\
Y_\alpha\ar^{p_\alpha}[r] & Y
}
$$
in $\DM(Y_\alpha,R)$ is an isomorphism:
$\derR p^*_\alpha\, \derR f_*\simeq \derR g_{\alpha,*}\, \derL q^*_\alpha$.

Then the base change morphism $\derL p^*\, \derR f_*\to \derR g_*\, \derL q^*$ is invertible in $\DM(Y',R)$.
\end{lm}
\begin{proof}
We want to prove that, for any object $E$ of $\DM(X,R)$, the map
$$\derL p^*\, \derR f_*(E)\to \derR g_*\, \derL q^*(E)$$
is invertible.
For this, it is sufficient to prove that, for any constructible object $M$
of $\DM(Y',R)$, the map
$$\Hom(M,\derL p^*\, \derR f_*(E))
\to\Hom(M,\derR g_*\, \derL q^*(E))$$
is bijective. Since $\DM(-,R)$ is continuous with respect to
dominant affine morphisms, we may assume that there exists an index
$\alpha_0$ and a constructible object $M_{\alpha_0}$, such that
$M\simeq\derL p^*_{\alpha_0}(M_{\alpha_0})$.
For $\alpha>\alpha_0$, we will write $M_\alpha$ for the
pullback of $M_{\alpha_0}$ along the transition map $Y_\alpha\to Y_{\alpha_0}$.
By continuity, we have a canonical identification
$$\varinjlim_\alpha\Hom(M_\alpha,\derL p^*_\alpha\, \derR f_*(E))
\simeq \Hom(M,\derL p^*\, \derR f_*(E))\, .$$
On the other hand, by assumption, we also have:
\begin{align*}
\varinjlim_\alpha\Hom(M_\alpha,\derL p^*_\alpha\, \derR f_*(E))
&\simeq\varinjlim_\alpha\Hom(M_\alpha,\derR g_{\alpha,*}\, \derL q^*_\alpha(E))\\
&\simeq\varinjlim_\alpha\Hom(\derL g^*_\alpha(M_\alpha),\derL q^*_\alpha(E))\, .\\
\end{align*}
The flatness of the maps $p_{\beta\alpha}$ ensures that the transition
maps of the projective system $(X_\alpha)$ are also affine and dominant, so that,
by continuity, we get the isomorphisms
\begin{align*}
\varinjlim_\alpha\Hom(\derL g^*_\alpha(M_\alpha),\derL q^*_\alpha(E))
&\simeq\Hom(\derL g^*(M),\derL q^*(E))\\
&\simeq\Hom(M,\derR g_*\, \derL q^*(E))\, ,
\end{align*}
and this achieves the proof.
\end{proof}

\begin{prop}\label{prop:DM_localiz_regular}
Let $i:Z \rightarrow S$ be a closed immersion between regular
 $k$-schemes. Assume that $S$ is the limit of a projective
 system of smooth separated $k$-schemes of finite type, with
 affine flat transition maps.
 Then $\DM(-,R)$ satisfies the localization property
 with respect to $i$ (cf. \cite[Def. 2.3.2]{CD3}).
\end{prop}
\begin{proof}
According to \cite[11.4.2]{CD3}, the proposition
 holds when $S$ is smooth of finite type over $k$
 -- the assumption then implies that $Z$ is smooth of finite
 type over $k$.

According to \cite[2.3.18]{CD3},
 we have only to prove that for any smooth $S$-scheme $X$,
 putting $X_Z \times_S Z$,
 the canonical map in $\DM(S,R)$
\begin{equation} \label{eq1:pf_loc_reg_DM}
M_S(X/X-X_Z) \rightarrow i_*\big(M_Z(X_Z)\big)
\end{equation}
is an isomorphism. This property is clearly local for the Zariski
topology, so that we can even assume that both $X$ and $S$ are
affine.

\bigskip

 Lifting the ideal of definition of $Z$, one can assume that
 $Z$ lifts to a closed subscheme $i_\alpha:Z_\alpha \hookrightarrow S_\alpha$.
 We can also assume that $i_\alpha$ is regular (apply \cite[9.4.7]{EGA4}
 to the normal cone of the $i_\alpha$).
 Thus $Z_\alpha$ is smooth over $k$.
Because $X/S$ is affine of finite presentation, it can be lifted 
 to a smooth scheme $X_\alpha/S_\alpha$ and because $X/S$ is smooth
 we can assume $X_\alpha/S_\alpha$ is smooth. 

Put $X_{Z,\alpha}=X_\alpha \times_{S_\alpha} Z_\alpha$.
Then, applying localization with respect to $i_\alpha$, we obtain that
 the canonical map:
\begin{equation} \label{eq2:pf_loc_reg_DM}
M_{S_\alpha}(X_\alpha/X_\alpha-X_{Z,\alpha})
 \rightarrow i_{\alpha*}(M_{Z_\alpha}\big(X_{Z,\alpha})\big)
\end{equation}
is an isomorphism in $\DM(S_\alpha,R)$.
Of course the analogue of \eqref{eq2:pf_loc_reg_DM}
remains an isomorphism for any $\alpha'>\alpha$.
Given $\alpha'>\alpha$, let us consider the cartesian square
$$
\xymatrix@=14pt{
Z_{\alpha'}\ar^{i_{\alpha'}}[r]\ar_{g}[d] & S_{\alpha'}\ar^{f}[d] \\
Z_\alpha\ar^{i_\alpha}[r] & S_\alpha
}
$$
in which $f:X_{\alpha'}\to X_\alpha$ denotes the transition map.
Then according to \cite[Prop. 2.3.11(1)]{CD3},
 the localization property with respect to $i_\alpha$
 and $i_{\alpha'}$ implies that the canonical base change map
$f^* i_{\alpha,*} \rightarrow i_{\alpha',*} g^*$
is an isomorphism.
By virtue of Lemma \ref{lm:basechangebycontinuity}, 
if $\varphi: S\to S_\alpha$ denote the canonical projection, the pullback
square
$$
\xymatrix@=14pt{
Z\ar^{i}[r]\ar_{\psi}[d] & S\ar^{\varphi}[d] \\
Z_\alpha\ar^{i_\alpha}[r] & S_\alpha
}
$$
induces a base change isomorphism $\derL\varphi^* i_{\alpha,*}\to i_*\derL\psi^*$.
Therefore, the image of the map \eqref{eq2:pf_loc_reg_DM} by $\derL\varphi^*$
is isomorphic to the map \eqref{eq1:pf_loc_reg_DM}, and this ends the proof.
\end{proof}

\subsection{Motives over fields}

This section is devoted to prove Theorem
 \ref{thm:comparisonnis} when one restricts
 to field extensions of $k$:
\begin{prop} \label{prop:fields}
Consider the assumptions of \ref{thm:comparisonnis}
 and let $K$ be an extension field of $k$.
Then the functor
$$
t^*:\HZ_K\Mod\ \rightarrow \DM(K,R)
$$
is an equivalence of symmetric monoidal triangulated categories.
\end{prop}
In the case where $K$ is a perfect field,
 this result is proved in \cite[5.8]{HKO} in a slightly different 
 theoretical setting. The proof will be given below (page \pageref{proof:fields}),
 after a few steps of preparation.

\begin{num} \label{num:elementary_traces}
In the end, the main theorem will prove the existence of very general
 trace maps, 
 but the proof of
 this intermediate result requires that we give a preliminary construction
 of traces in the following case.

Let $K$ be an extension field of $k$, and $f:Y \rightarrow X$ be a flat
finite surjective morphism of degree $d$
between integral $K$-schemes.
There is a natural morphism $^tf:R_X\to f_\sharp(R_Y)$ in $\uDM(X,R)$,
defined by the transposition of the graph of $f$.
The composition
$$f_\sharp(R_Y)\to R_X\xrightarrow{^tf}f_\sharp(R_Y)$$
is $d$ times the identity of $f_\sharp(R_Y)$;
see \cite[Prop. 9.1.13]{CD3}.
Moreover, if $f$ is radicial (i.e. if the field of functions on $Y$ is a purely
inseparable extension of the field of functions of $X$), then
the composition
$$R_X\xrightarrow{^tf}f_\sharp(R_Y)\xrightarrow{f}R_X$$
is $d$ times the identity of $R_X$;
see \cite[Prop. 9.1.14]{CD3}.
In other words, in the latter case, since $p$ is invertible, the co-unit map
$f_\sharp(R_Y)\to R_X$ is an isomorphism in $\uDM(X,R)$.
\end{num}

\begin{lm}\label{lm:degreeformula0}
Under the assumptions of the previous paragraph, if $f$ is radicial, then
the pullback functor
$$\derL f^*:\DM(X,R)\to\DM(Y,R)$$
is fully faithful.
\end{lm}

\begin{proof}
As the inclusion $\DM(-,R)\subset\uDM(-,R)$ is fully faithful
and commutes with $\derL f^*$, it is sufficient to prove that the functor
$$f^*:\uDM(X,R)\to\uDM(Y,R)$$
is fully faithful. In other words, we must see that the composition of $f^*$
with its left adjoint $f_\sharp$ is isomorphic to the identity functor
(through the co-unit of the adjunction).
For any object $M$ of $\uDM(X,R)$, we have a projection formula:
$$f_\sharp f^*(M)\simeq f_\sharp(R_Y)\otimes^\derL_R M\, .$$
Therefore, it is sufficient to check that the co-unit
$$f_\sharp(R_Y)\simeq R_X$$
is an isomorphism. Since $f$ is radicial, its degree must be
a power of $p$, hence must be invertible in $R$.
An inverse is provided by the map $^tf:R_X\to f_\sharp(R_Y)$.
\end{proof}

\begin{num} \label{num:elementary_traces2}
These computations can be interpreted in terms of $\HZ$-modules
as follows (we keep the assumptions of \ref{num:elementary_traces}).

Using the internal Hom of $\uDM(X,R)$, one gets
a morphism
$$Tr_f:\derR f_*(R_Y)\to R_X$$
Since the right adjoint of the inclusion $\DM(-,R)\subset\uDM(-,R)$
commutes with $\derR f_*$, the map $Tr_f$ above can be seen as
a map in $\DM(X,R)$. Similarly, since the functor
$t_*:\DM(-,R)\to\HZ\Mod$ commutes with $\derR f_*$, we get a trace morphism
$$Tr_f:\derR f_*\HZ_Y\to\HZ_X$$
in $\HZ_X\Mod$. For any $\HZ_X$-module $E$, we obtain a trace
morphism
$$Tr_f:\derR f_*\derL f^*(E) \to E$$
as follows. Since we have the projection formula
$$\derR f_*(\HZ_Y)=\derR f_*\derL f^*(\HZ_X)\simeq\derR f_*(\un_Y)\otimes^\derL\HZ_X\, ,$$
the unit $\un_X\to\HZ_X$ induces a map
$$\widetilde{Tr}_f:\derR f_*(\un_Y)\to\derR f_*(\un_Y)\otimes^\derL\HZ_X\simeq\derR f_*\derL f^*(\HZ_X)
\simeq\derR f_*(\HZ_Y)\xrightarrow{Tr_f}\HZ_X\, .$$
For any $\HZ_X$-module $E$, tensoring the map $\widetilde{Tr}_f$ with identity
of $E$ and composing with the action $\HZ_X\otimes^\derL E\to E$
leads to a canonical morphism in $\HZ_X\Mod$:
$$Tr_f:\derR f_*\derL f^*(E)\simeq\derR f_*(\un_Y)\otimes^\derL E\to E\, .$$
By construction of these trace maps, we have the following lemma.
\end{num}

\begin{lm}\label{lm:formuledegre}
Under the assumptions of paragraph \ref{num:elementary_traces},
for any $\HZ_X$-module $E$,
the composition of $Tr_f$ with the unit of
the adjunction between $\derL f^*$ and $\derR f_*$
$$E\to\derR f_*\derL f^*(E)\xrightarrow{Tr_f}E$$
is $d$ times the identity of $E$.
If, moreover, $f$ is radicial, then the composition
$$\derR f_*\derL f^*(E)\xrightarrow{Tr_f}E\to
\derR f_*\derL f^*(E)$$
is also $d$ times the identity of $\derR f_*\derL f^*(E)$.
\end{lm}

This also has consequences when looking at the $\HZ_K$-modules
associated to $X$ and $Y$. To simplify the notations, we will write
$$\HZ(U)=\HZ_K\otimes^\derL\Sigma^\infty(U_+)$$
for any smooth $K$-scheme $U$.

\begin{lm}\label{lm:formuledegre2}
Under the assumptions of paragraph \ref{num:elementary_traces},
if $d$ is invertible in $R$, and if both $X$ and $Y$
are smooth over $K$, then $\HZ(X)$
is a direct factor of $\HZ(Y)$ in $\HZ_K\Mod$.
\end{lm}

\begin{proof}
Let $p:X\to\Spec(K)$ and $q:Y\to\Spec(K)$ be the structural maps
of $X$ and $Y$, respectively. Since $pf=q$, for any $\HZ_K$-module $E$, we have:
\begin{align*}
\Hom(\HZ(X),E)
&=\Hom(\HZ_X,p^*(E))\\
\Hom(\HZ(Y),E)
&=\Hom(\HZ_X,\derR f_*\derL f^*p^*(E))\, .
\end{align*}
Therefore, this lemma is a translation of the first assertion of Lemma \ref{lm:formuledegre} and of the Yoneda Lemma.
\end{proof}

\begin{proof}[Proof of Proposition \ref{prop:fields}]\label{proof:fields}
We first consider the case of a perfect field $K$.
The reference is \cite[5.8]{HKO}.
 We use here a slightly different theoretical setting than these authors
 so we give a proof to convince the reader.

Because $t^*$ preserves the canonical compact
generators of both categories, we need only
 to prove it is fully faithful on a family
 of compact generators of $\HZ_K\Mod$ (see \cite[Corollary 1.3.21]{CD3}).
 For any $\HZ_K$-modules $E$, $F$ belonging to a suitable
 generating family of $\HZ_K\Mod$, and
 and any integer $n$, we want to prove that the map
\begin{equation} \label{eq:pf_thm:fields}
\Hom_{{\HZ_K}\Mod}(E,F[n])
 \xrightarrow{t^*} \Hom_{\DM(K,R)}(t^*(E),t^*(F)[n])
\end{equation} 
For this purpose, using the method of \cite[Sec. 1]{RiouCRAS},
with a small change indicated below,
 we first prove that $\HZ_K\Mod$ is generated by objects
 of the form $\HZ(X)(i)$ for a smooth projective $K$-scheme $X$
 and an integer $i$. Since these are compact, it is sufficient to prove the following
 property: for any $\HZ_K$-module $M$ such that
 $$\Hom_{\HZ_K\Mod}(\HZ(X)(p)[q],M)=0$$
 for any integers $p$ and $q$, we must have $M\simeq 0$.
 To prove the vanishing of $M$, it is sufficient to prove the vanishing
 of $M\otimes\ZZ_{(\ell)}$ for any prime $\ell\neq p$.
 On the other hand, for any compact object $C$, the formation of $\Hom(C,-)$ commutes with tensoring by $\ZZ_{(\ell)}$; therefore, we may assume
 $R$ to be a $\ZZ_{(\ell)}$-algebra for some prime number $\ell\neq p$.
 Under this additional assumption, we will prove that,
 for any smooth connected $K$-scheme $X$,
 the object $\HZ(X)=\HZ_k\otimes^\derL\Sigma^\infty(X_+)$
 is in the thick subcategory $\mathcal P$ generated by
 Tate twists of $\HZ_K$-modules of the form $\HZ(W)$ for $W$
 a smooth projective $K$-scheme. Using the induction
 principle explained by Riou in \emph{loc. cit.}, on the dimension $d$ of $X$,
 we see that, given any couple $(Y,V)$, where $Y$ is a smooth
 $K$-scheme of dimension $d$, and $V$ is a dense open subscheme of $Y$, the property that $\HZ(Y)$ belongs to $\mathcal P$ is equivalent to the property $\HZ(V)$ belongs to $\mathcal P$.
 Therefore, it is enough to consider the case of a dense open subscheme of $X$
 which we can shrink at will.
 In particular, applying Gabber's theorem \cite[IX, 1.1]{gabber},
  we can assume there exists a flat, finite, and surjective morphism,
  $f:Y \rightarrow X$ which is of degree prime to $\ell$, and
  such that $Y$ is a dense open subscheme of a smooth
  projective $k$-scheme. Since $\HZ(Y) \in \mathcal P$, Lemma \ref{lm:formuledegre2}
  concludes.

We now are reduced to prove that the map \eqref{eq:pf_thm:fields} is an isomorphism
 when $E=\HZ(X)(i)$ and $F=\HZ(Y)(j)$ for $X$ and $Y$ smooth and projective over $K$.
 Say $d$ is the dimension of $Y$. Then according to \cite[Sec. 5.4]{Deg8},
 $\HZ_K(Y)$ is strongly dualizable with strong dual $\HZ_K(Y)(-d)[-2d]$.
 Then the result follows from the fact that the two members of
 \eqref{eq:pf_thm:fields} compute the motivic cohomology
 group of $X \times_K Y$ in degree $(n-2d,j-i-d)$ (in a compatible way,
 because the functor $t^*$ is symmetric monoidal).
 This achieves the proof of Proposition \ref{prop:fields}
 in the case where the ground field $K$ is perfect.

\bigskip

Let us now consider the general case. Again, we are reduced
 to prove the map \eqref{eq:pf_thm:fields} is an isomorphism
 whenever $E$ and $F$ are compact (hence constructible).
Let $K$ be a finite extension of $k$, and
let $L/K$ be a finite totally inseparable extension of fields,
with corresponding morphism of schemes
$f:\Spec(L) \rightarrow \Spec(K)$.
According to Lemma \ref{lm:degreeformula0}, the functor
$\derL f^*:\DM(K,R) \rightarrow \DM(L,R)$ is fully
 faithful. Moreover, the pullback functor
 $\derL f^*:\HZ_K\Mod\  \rightarrow \HZ_L\Mod\ $
 is fully faithful as well; see the last assertion of Lemma \ref{lm:formuledegre}
 (and recall that the degree of the extension $L/K$ must be a power of $p$,
 whence is invertible in $R$).
Thus, by continuity of the premotivic categories
 $\DM(-,R)$ and $\HZ\Mod$ (see Examples \ref{ex:continuity}(2) and
 \ref{ex:modules_reg}), Proposition \ref{prop:continuity&2lim}
 gives the following useful lemma:
\begin{lm}\label{lm:insepclosure}
Let $K^s$ be the inseparable closure of $K$ (i.e. the
biggest purely inseparable extension of $K$ in some algebraic
closure of $K$). Then the following pullback functors:
$$
\DM_c(K,R) \rightarrow \DM_c(K^s,R)\quad\text{and}\quad
 \ \HZ_K\Mod_c  \rightarrow \HZ_{K^s}\Mod_c
$$
are fully faithful.
\end{lm}
With this lemma in hands, to prove that \eqref{eq:pf_thm:fields} is an isomorphism for constructible $\HZ_K$-modules $E$ and $F$, we can replace the field $K$
by the perfect field $K^s$, and this proves Proposition \ref{prop:fields} in full generality.
\end{proof}

\subsection{Proof in the regular case}\label{subsec:proofregcase}

In the course of the proof of Theorem \ref{thm:comparisonnis}, we wil use the following
 lemma:
\begin{lm}\label{lm:oubliDMversHZmod}
Let $T$ and $S$ be regular $k$-schemes
 and $f:T \rightarrow S$ be a morphism of $k$-schemes.
\begin{enumerate}
\item If $T$ is the limit of a projective system of $S$-schemes
with dominant affine smooth transition morphisms,
 then $t_*$ commutes with $f^*$.
\item If $f$ is a closed immersion, and if $S$
is the limit of a projective system of smooth separated
$k$-schemes of finite type with flat affine transition morphims,
 then $t_*$ commutes with $f^*$.
\item If $f$ is an open immersion,
 then $t_*$ commutes with $f_!$.
\end{enumerate}
\end{lm}
\begin{proof}
The forgetful functor $\mathcal O_{\HZ}:\HZ\Mod\to\SH$ is conservative,
 and it commutes with $f^*$ for any morphism $f$
 and with $j_!$ for any open immersion; see the last assertion of \cite[Prop. 7.2.14]{CD3}.
 Therefore, it is sufficient to check each case of this lemma after replacing $t_*$ by $\varphi_*$.

Then, case (1) follows easily by continuity of $\DM$ and $\SH$
 with respect to dominant maps, and from the case where $f$ is a smooth morphism. Case (2) was proved in Proposition \ref{prop:DM_localiz_regular}.
 (taking into account \ref{prop:localization&Voevodsky_conj}).
 Then case (3) finally follows from results of \cite{CD3}:
 in fact $\varphi_*$ is defined as the following composition:
$$
\DM(S,R) \xrightarrow{\derL \gamma_*} \Der_{\AA^1}(S,R)
 \xrightarrow K \SH(S)
$$
with the notation of \cite[11.2.16]{CD3} ($\Lambda=R$).
 The fact $K$ commutes with $j_!$ is obvious
 and for $\derL \gamma_*$, this is \cite[6.3.11]{CD3}.
\end{proof}

To be able to use the refined version of Popescu's theorem
 proved by Spivakovsky
 (see \cite[Th. 10.1]{spivak}, ``resolution
 by smooth sub-algebras''), we will need the following 
 esoteric tool extracted from an appendix of Bourbaki
 (see \cite[IX, Appendice]{AC89} and, in particular,
 Example 2).

\begin{df}\label{df:gonflement}
Let $A$ be a local ring with maximal ideal $\mathfrak m$.

We define the \emph{$\infty$-gonflement} (resp. \emph{$n$-gonflement})
 of $A$ as the localization of the polynomial $A$-algebra
 $A\lbrack(x_i)_{i \in \NN}\rbrack$
 (resp. $A\lbrack(x_i)_{0 \leq i \leq n}\rbrack$)
 with respect to the prime ideal $\mathfrak m.A[x_i, i \in \NN]$
 (resp. $\mathfrak m.A[x_i, 0 \leq i \leq n]$).
\end{df}

\begin{num}\label{num:gonflement}
Let $B$ (resp. $B_n$) be the $\infty$-gonflement (resp. $n$-gonflement)
 of a local noetherian ring $A$. We will use the following facts
 about this construction, which are either obvious or follow
 from \emph{loc. cit.}, Prop. 2:
\begin{enumerate}
\item The rings $B$ and $B_n$ are noetherian.
\item The $A$-algebra $B_n$ is the localization of a smooth $A$-algebra.
\item The canonical map $B_n \rightarrow B_{n+1}$ is injective.
\item $B=\ilim_{n \in \NN} B_n$, with the obvious transition maps.
\end{enumerate}
We will need the following easy lemma:
\begin{lm}
Consider the notations above. Assume that $A$ is a local henselian
 ring with infinite residue field. Then for any integer $n \geq 0$,
 the $A$-algebra $B_n$ is a filtered inductive limit
 of its smooth and split sub-$A$-algebras.
\end{lm}
\begin{proof}
We know that $B_n$ is the union of $A$-algebras of the form
 $A[x_1,\hdots,x_n][1/f]$ for a polynomial
 $f \in A[x_1,...,x_n]$
 whose reduction modulo $\mathfrak m$ is non zero.
Let us consider the local scheme $X=\Spec(A)$,
 $s$ be its closed point and put
 $U_n(f)=\Spec(A[x_1,\hdots,x_n][1/f])$ for a polynomial $f$ as above.
 To prove the lemma, it is sufficient to prove that $U_n(f)/X$
 admits a section. By definition, the fiber $U_n(f)_s$ of $U_n(f)$ at 
 the point $s$ is a non empty open subscheme. 
As $\kappa(s)$ is infinite by assumption,
 $U_n(f)_s$ admits a $\kappa(s)$-rational point.
 Thus $U_n(f)$ admits an $S$-point because $X$ is henselian
 and $U_n(f)/X$ is smooth (see \cite[18.5.17]{EGA4}).
\end{proof}
Combining properties (1)-(4) above with the preceding lemma,
 we get the following property:
\begin{itemize}
\item[(G)] Let $A$ be a noetherian local henselian ring
 with infinite residue field,
 and $B$ be its $\infty$-gonflement.
 Then $B$ is a noetherian $A$-algebra which
  is the filtering union of a family $(B_\alpha)_{\alpha \in I}$
	of smooth split sub-$A$-algebras of $B$.
\end{itemize}
\end{num}

\begin{lm}\label{lm:gonfleconserve}
Consider the notations of property (G).
 Then the pullback along the induced map
$p:X'=\Spec(B)\to X=\Spec(A)$ defines a conservative
functor $\derL p^*:\SH(X)\to\SH(X')$.
\end{lm}

\begin{proof}
Let $E$ be an object of $\SH(X)$ such that $\derL p^*(E)=0$
in $\SH(X')$. We want to prove that $E=0$.
For this, it is sufficient to prove that, for any
constructible object $C$ of $\SH(X)$, we have
$$\Hom(C,E)=0\, .$$
Given the notations of property (G),
 and any index $\alpha \in I$,
 let $C_i$ and $E_i$ be the respective pullbacks of $C$ and $E$
 along the structural map $p_\alpha:\Spec(B_\alpha) \rightarrow \Spec(A)$.
Then, by continuity, the map
$$
\varinjlim_\alpha\Hom(C_\alpha,E_\alpha)
 \to \Hom(\derL p^*(C),\derL p^*(E))
$$
is an isomorphism, and thus, according to property (G),
the map
$$
\Hom(C,E)\to\Hom(\derL p^*(C),\derL p^*(E))
$$
 is injective because each map $p_\alpha$ is a split
 epimorphism.
\end{proof}

In order to use $\infty$-gonflements in $\HZ$-modules
without any restriction on the size of the ground field, we will need the
the following trick, which makes use of transfers up to homotopy:
\begin{lm}\label{lm:transcendental_split_DM}
Let $L/K$ be a purely transcendental
extension of fields of transcendence degree $1$, with $K$ perfect, and let
$p:\Spec(L)\to\Spec(K)$ be the induced morphism of schemes.
Then, for any objects $M$ and $N$ of $\DM(K,R)$, if $M$ is compact, then
the natural map
$$\Hom_{\DM(K,R)}(M,N)\to\Hom_{\DM(K,R)}(M,\derR p_*\,p^*(N))=\Hom_{\DM(K,R)}(\derL p^*(M),\derL p^*(N))$$
is a split embedding. In particular, the pullback functor
$$\derL p^*:\DM(K,R)\to\DM(L,R)$$
is conservative.
\end{lm}
\begin{proof}
Let $I$ be the cofiltering set of affine open neighbourhoods
 of the generic point of $\AA^1_K$ ordered by inclusion.
 Obviously, $\Spec(L)$ is the projective limit of these open neighbourhoods.
 Thus, using continuity for $\DM$ with respect to dominant maps,
 we get that:
$$
\Hom(M,\derR p_*\,\derL p^*(N))=
 \varinjlim_{V \in I^{op}}\Hom(M(V),\uHom(M,N))\, .
$$
We will use the language of generic motives
 from \cite{Deg5}. Recall that $M(L)=\text{``$\varprojlim M(V)$''}$
is a pro-motive in $\DM(K)$, so that the preceding identification now takes
the following form.
$$\Hom(M,\derR p_*\,\derL p^*(N))=\Hom(M(L),\uHom(M,N))\, .$$
 Since, according to \cite[Cor. 6.1.3]{Deg5}, the canonical
 map $M(L) \rightarrow M(K)$ is a split epimorphism of pro-motives,
this proves the first assertion of the lemma. The second assertion is a direct consequence
of the first and of the fact that the triangulated category
$\DM(K,R)$ is compactly generated.
%
\end{proof}

\begin{proof}[Proof of Theorem \ref{thm:comparisonnis}]
We want to prove that for a regular noetherian $k$-scheme
of finite dimension $S$, the adjunction:
$$
t^*:\HZ_S\Mod\ \rightleftarrows \DM(S,R):t_*
$$
is an equivalence of triangulated categories.
 Since the functor $t^*$ preserves compact objects, and since
 there is a generating family of compact
 objects of $\DM(S,R)$ in the essential image of the functor $t^*$,
 it is sufficient to prove that $t^*$ is fully faithful
 on compact objects (see \cite[Corollary 1.3.21]{CD3}): we have to prove that,
 for any compact $\HZ_S$-module $M$,
 the adjunction map
 $\eta_M: M \rightarrow t_*t^*(M)$ is an isomorphism.

\smallskip

\noindent \emph{First case}: We first assume that $S$ is essentially
 smooth -- \emph{i.e.} the localization of a smooth $k$-scheme.
 We proceed by induction on the dimension of $S$.
 The case of dimension $0$ follows from Proposition \ref{prop:fields}.

We recall that the category $\HZ\Mod$ is continuous on $\reg$
 (\ref{prop:HZmodpremotivicReg}).
 Let $x$ be a point of $S$ and $S_x$ be the localization of $S$
 at $x$, $p_x:S_x \rightarrow S$ the natural projection.
 Then it follows from \cite[Prop. 4.3.9]{CD3} that the family of functors:
$$
p_x^*:\HZ_S\Mod \rightarrow \HZ_{S_x}\Mod, x \in S
$$
is conservative.

Since $p_x^*$ commutes with $t^*$ (trivial) and with $t_*$
(according to Lemma \ref{lm:oubliDMversHZmod}),
we can assume that $S$ is
a local essentially smooth $k$-scheme.

To prove the induction case, let $i$ (resp. $j$) be the immersion
 of the closed point $x$ of $S$ (resp. of the open complement $U$
 of the closed point of $S$).
Since the localization property with respect to $i$
 is true in $\HZ\Mod$ (because it is true in $\SH$, using
 the last assertions of Proposition \ref{prop:HZmodpremotivicReg})
 and in $\DM$ (because of Proposition \ref{prop:DM_localiz_regular}
  that we can apply because we have assumed that
	$S$ is essentially smooth),
 we get two morphisms of distinguished triangles:
$$
\xymatrix@=10pt{
j_!j^*(M)\ar[r]\ar[d] & M\ar[r]\ar[d] & i_*i^*(M)\ar[d]\ar[r] &j_!j^*(M)[1]\ar[d] \\
j_!j^*(t_*t^*(M))\ar[r]\ar^\wr[d] & t_*t^*(M)\ar[r]\ar@{=}[d]
 & i_*i^*(t_*t^*(M))\ar^\wr[d]\ar[r] & j_!j^*(t_*t^*(M))[1]\ar[d]^\wr\\
j_!t_*t^*j^*(M)\ar[r] & t_*t^*(M)\ar[r] & i_*t_*t^*i^*(M)\ar[r] &j_!t_*t^*j^*(M)[1]
}
$$
The vertical maps on the second floor are isomorphisms:
both functors $t^*$ and $t_*$ commute with $j^*$ (as $t^*$ is the
left adjoint in a premotivic adjunction, it commutes with $j_!$ and $j^*$,
and this implies that $t_*$ commutes with $j^*$, by transposition);
the functor $t^*$ commutes with $i_*$ because it commutes with $j_!$, $j^*$
and $i^*$, and because the localization property with respect to $i$
is verified in
$\HZ\Mod$ as well as in $\DM$); finally, applying the third assertion of Lemma \ref{lm:oubliDMversHZmod} for $f=j$,
this implies that the functor $t_*$ commutes with $i^*$.
 To prove that the map $\eta_M$ is an isomorphism,
 it is thus sufficient to treat the case of $j_!\eta_{j^*(M)}$
 and of $i_*\eta_{i^*(M)}$.
 This means we are reduced to the cases of $U$ and $\Spec(\kappa(x))$,
 which follow respectively from the inductive assumption
 and from the case of dimension zero.

\smallskip

\noindent \emph{General case}: Note that the previous case
 shows in particular the theorem for any smooth $k$-scheme.
 Assume now that $S$ is an arbitrary regular noetherian $k$-scheme. 
Using \cite[Prop. 4.3.9]{CD3} again,
and proceeding as we already did above
(but considering limits of Nisnevich neighbourhoods instead of
Zariski ones), we may assume that $S$ is henselian.
Let $L=k(t)$ be the field of rational functions, and let us form the following
pullback square.
$$\xymatrix{
S'\ar[r]^q\ar[d]_g&S\ar[d]^f\\
\Spec(k(t))\ar[r]^p&\Spec(k)
}$$
Then the functor
$$\derR p_*\, \derL p^*:\HZ_{k}\Mod\to\HZ_{k}\Mod$$
is conservative: this follows right away from Lemma \ref{lm:transcendental_split_DM}
and Proposition \ref{prop:fields}.
This implies that the functor
$$\derL q^*:\HZ_S\Mod\to\HZ_{S'}\Mod$$
is conservative. To see this, let us consider an object $E$ of $\HZ_S\Mod$
such that $\derL q^*(E)=0$. To prove that $E=0$, it is sufficient to prove that
$\Hom(M,E)=0$ for any compact object $M$ of $\HZ_S\Mod$.
Formula
$$\Hom(\HZ_k,\derR f_*\uHom(M,E))
\simeq\Hom(M,E)$$
implies that
it is sufficient to check that $\derR f_*\uHom(M,E)=0$ for any compact object $M$
(where $\uHom$ is the internal Hom of $\HZ_S\Mod$).

Since the functor $\derR p_*\, \derL p^*$
is conservative, it is thus sufficient to prove that
$$\derR p_*\, \derL p^*\derR f_*\uHom(M,E)=0\, .$$
We thus conclude with the following computations (see \cite[Propositions 4.3.11 and 4.3.14]{CD3}).
\begin{align*}
\derR p_*\, \derL p^*\derR f_*\uHom(M,E)
&\simeq\derR p_*\, \derR g_*\, \derL q^*\uHom(M,E)\\
&\simeq\derR p_*\, \derR g_*\uHom(\derL q^*(M),\derL q^*(E))=0\\
\end{align*}
In conclusion, since the functor $\derL q^*$ commutes with $t_*$
(see Lemma \ref{lm:oubliDMversHZmod}~(1)), we may replace $S$ by $S'$
and thus assume that the residue field of $S$ is infinite.
 Let $B$ be the $\infty$-gonflement of $A=\Gamma(S,\cO_S)$ (Definition \ref{df:gonflement}),
 and $f:T=\Spec(B)\to S$ be the map induced by the inclusion $A\subset B$.
 We know that the functor
 $$\derL f^*:\HZ_S\Mod\to\HZ_T\Mod$$
 is conservative: as the forgetful functor $\HZ\Mod\to\SH$ is
 conservative and commutes with $\derL f^*$, this follows from
 Lemma \ref{lm:gonfleconserve} (or one can reproduce the proof
 of this lemma, which only used the continuity property of $\SH$
 with respect to projective systems of schemes with
 dominant affine transition morphisms).
 Similarly, it follows again from Lemma \ref{lm:oubliDMversHZmod}~(1)
 that the functor $t_*$ commutes with $\derL f^*$.
 As the functor $t^*$ commutes with $\derL f^*$, it is sufficient to prove
 that the functor $t^*$ is fully faithful over $T$, and it is still sufficient to
 check this property on compact objects.
 Since the ring $B$ is noetherian and regular, and has a field of functions with
 infinite transcendance degree over the perfect field $k$ (see \ref{num:gonflement}),
 it follows from Spivakovsky's refinement of Popescu's Theorem~\cite[10.1]{spivak}
 that $B$ is the filtered union of its smooth subalgebras of finite type
 over $k$. In other terms,
 $T$ is the projective limit of a projective system of smooth affine $k$-schemes
 of finite type $(T_\alpha)$
 with dominant transition maps. Therefore, by continuity
 (see Examples \ref{ex:modules_reg} and \ref{ex:continuity}(2)),
 we can apply Proposition \ref{prop:continuity&2lim} twice
 and see that the functor
 $$2\text{-}\varinjlim_\alpha\HZ_{T_\alpha}\Mod_c\simeq\HZ_T\Mod_c
 \to 2\text{-}\varinjlim_\alpha\DM_c(T_\alpha,R)\simeq\DM_c(T,R)$$
 is fully faithful, as a filtered $2$-colimit of functors having this property.
\end{proof}

\section{More modules over motivic Eilenberg-MacLane spectra}

\begin{num}\label{num:positive}
Given a scheme $X$, let $\mathit{Mon}(X)$ be the category of
unital associative monoids in the category of symmetric Tate spectra $\Sp_X$.
The forgetful functor
$$U:\mathit{Mon}(X)\to\Sp_X$$
has a left adjoint, the free monoid functor:
$$F:\Sp_X\to\mathit{Mon}(X)\, .$$
Since the stable model category of symmetric Tate spectra satisfies
the monoid axiom (see \cite[Lemma 4.2]{hoyois2}),
by virtue of a well known theorem of Schwede and Shipley \cite[Theorem 4.1(3)]{SS},
the category $\mathit{Mon}(X)$ is endowed with a combinatorial model category structure
whose weak equivalences (fibrations) are the maps whose image by $U$
are weak equivalences (fibrations) in $\Sp_X$;
furthermore, any cofibrant monoid is also cofibant as an object of $\Sp_X$.
\end{num}

\begin{num}\label{defineHZrelatif}
We fix once and for all a cofibrant resolution
$$\HZ'\to\HZ_k$$
of the motivic Eilenberg-MacLane spectrum $\HZ_k$ in the
model category $\mathit{Mon}(k)$. Given a $k$-scheme $X$ with
structural map $f:X\to\Spec(k)$, we define
$$\HZ_{X/k}=f^*(\HZ')$$
(where $f^*$ denotes the pullback functor in the premotivic model
category $\Sp$). The family $(\HZ_{X/k})_X$ is a cartesian
section of the $\sm$-fibred category of monoids in $\Sp$ which
is also homotopy cartesian (as we have an equality
$\derL f^*(\HZ_k)=\HZ_{X/k}$).
We write $\HZ_{X/k}\Mod$ for the homotopy category of
(left) $\HZ_{X/k}$-modules.

This notation is in conflict with
the one introduced in Definition \ref{df:HZrelatif}.
This conflict disappears up to weak equivalence\footnote{In
the proof of Theorem \ref{thm:comparisonnis}, we used the fact
that the spectra $\HZ_{X/k}$, as defined in Definition \ref{df:HZrelatif}, are \emph{commutative}
monoids of the model category
of symmetric Tate spectra (because we used
Poincar\'e duality in an essential way, in the
case where $X$ is the spectrum of a perfect field).
This new version of motivic Eilenberg-MacLane spectra $\HZ_{X/k}$ is not
required to be commutative anymore (one could force this property by working with fancier model categories of motivic
spectra (some version of the `positive model structure',
as discussed in \cite{hornbostel} for instance),
but these extra technicalities
are not necessary for our purpose.
We shall use Theorem \ref{thm:comparisonnis} in a crucial way, though.}:
when $X$ is regular, the comparison map
$$f^*(\HZ')\to f^*(\HZ_k)$$
is a weak equivalence (Proposition \ref{prop:exactness_regular}).
For $X$ regular, $\HZ_{X/k}$ is thus
a cofibrant resolution of $\HZ_X$
in the model category $\mathit{Mon}(X)$. In particular,
in the case where $X$ is regular, we have a canonical
equivalence of triangulated categories:
$$\HZ_{X/k}\Mod\simeq\HZ_{X}\Mod\, .$$
\end{num}
%
%

\begin{prop}\label{prop:HZmodrelatif2}
The assignment $X\mapsto\HZ_{X/k}\Mod$
defines a motivic category over the category of noetherian
$k$-schemes of finite dimension
which has the property of continuity with respect to
arbitrary projective systems with affine
transition maps.
Moreover, when we let $X$ vary,
both the free $\HZ_{X/k}$-algebra (derived) functor
$$L_{\HZ_{X/k}}:\SH(X)\to\HZ_{X/k}\Mod$$
and its right adjoint
$$\cO_{\HZ_{X/k}}:\HZ_{X/k}\Mod\to\SH(X)$$
are morphisms of premotivic triangulated categories
over the category of $k$-schemes.
In other words both functors commute with $\derL f^*$
for any morphism of $k$-schemes $f$, and with $\derL g_\sharp$
for any separated smooth morphism of $k$-schemes $g$.
\end{prop}

\begin{proof}
The first assertion comes from \cite[7.2.13 and 7.2.18]{CD3},
the one about continuity is a direct application of
Lemma \ref{lm:trivial_continuity}, and the last one
comes from \cite[7.2.14]{CD3}.
\end{proof}

\begin{rem}
Since the functor $\cO_{\HZ_{X/k}}:\HZ_{X/k}\Mod\to\SH(X)$
is conservative and preserves small sums, the family
of objects of the form $\HZ_{X/k}\otimes^\derL\Sigma^\infty(Y_+)(n)$,
for any separated smooth $X$-scheme $Y$ and any integer $n$, do form
a generating family of compact objects. In particular, the notions
of constructible object and of compact object coincide in $\HZ_{X/k}\Mod$ (see for instance
\cite[Remarks 5.4.10 and 5.5.11]{CD4},
for a context in which these two notions fail to coincide).
\end{rem}

\begin{num}\label{def:adjHZModDMcdh}
For any $k$-scheme $X$, we have canonical morphisms
of monoids in $\Sp_X$:
$$\HZ_{X/k}\to f^*(\HZ_k)\to\HZ_X\, .$$
In particular, we have a canonical functor
$$\HZ_{X/k}\Mod\to\HZ_X\Mod\ , \quad E\mapsto\HZ_X\otimes^\derL_{\HZ_{X/k}}E\, .$$
If we compose the latter with the functor
$$\HZ_X\Mod\xrightarrow{t^*}\DM(X,R)
\xrightarrow{\derL\rho_!}\uDM(X,R)
\xrightarrow{a^*_\cdh}\uDM_{\cdh}\, ,$$
we get a functor
$$\HZ_{X/k}\Mod\to\uDM(X,R)$$
which defines a morphism a premotivic categories.
In particular, this functor takes it values in $\DM_\cdh(X,R)$,
and we obtain a functor
$$\tau^*:\HZ_{X/k}\Mod\to\DM_\cdh(X,R)\, .$$
As $\tau^*$ preserves small sums, it has a right adjoint $\tau_*$,
and we finally get a premotivc adjunction
$$\tau^*:\HZ_{(-)/k}\Mod\leftrightarrows\DM_\cdh(-,R):\tau_*\, .$$
Moreover, the functor $\tau^*$ preserves the canonical
generating families of compact objects. Therefore, the functor $\tau_*$
is conservative and commutes with small sums.
\end{num}

\section{Comparison theorem: general case}

The aim of this section is to prove:

\begin{thm}\label{thm:comparisoncdh}
Let $k$ be a perfect field of characteristic exponent $p$.
Assume that $p$ is invertible in the ring of coefficients $R$.
For any noetherian $k$-scheme of finite dimension $X$, the canonical
functor
$$\tau^*:\HZ_{X/k}\Mod\to\DMcdh(X,R)$$
is an equivalence of categories. 
\end{thm}

The proof will take the following path: we will
prove this statement in the case where $X$ is
separated and of finite type over $k$.
For this, we will use Gabber's refinement
of de~Jong's resolution of singularities by alterations, as well
as descent properties for $\HZ_k$-modules proved by Shane Kelly
to see that it is sufficient to consider the case of
a smooth $k$-scheme. In this situation, Theorem \ref{thm:comparisoncdh}
will be a rather formal consequence of Theorem \ref{thm:comparisonnis}.
The general case will be obtained by a continuity argument.

\begin{paragr}
Let $\ell$ be a prime number.
Following S.~Kelly \cite{kelly}, one defines the $\ldh$-topology on the category of
noetherian schemes
as the coarsest Grothen\-dieck topology such that any $\cdh$-cover is
an $\ldh$-cover and any morphism of the form $f:X\to Y$, with $f$ finite, surjective, flat,
and of degree prime to $\ell$ is an $\ldh$-cover.
For instance, if $\{U_i\to X\}_{i\in I}$ is a $\cdh$-cover, and if, for
each $i$ one has a finite surjective flat morphism $V_i\to U_i$ of degree prime
to $\ell$, we get an $\ldh$-cover $\{V_i\to X\}_{i\in I}$.
In the case where $X$ is noetherian, one can show that, up to refinement,
any $\ldh$-cover is of this form; see \cite[Prop. 3.2.5]{kelly}.
We will use several times the following non-trivial fact, which
is a direct consequence of Gabber's
theorem of uniformization prime to $\ell$ \cite[Exp.~IX, Th.~1.1]{gabber}:
locally for the $\ldh$-topology, any quasi-excellent scheme is regular.
In other words, for any noetherian quasi-excellent scheme $X$
(e.g. any scheme of finite type over  field), there exists a morphism
of finite type $p:X'\to X$ which is a covering
for the $\ldh$-topology and has a regular domain.
\end{paragr}

\begin{prop}\label{prop:ldhdescenttransfers}
Let $F$ be a $\cdh$-sheaf with transfers over $X$ which is
$\mathbf{Z}_{(\ell)}$-linear. Then $F$ is an $\ldh$-sheaf and,
for any integer $n$, the map
$$H^n_\cdh(X,F)\to H^n_\ldh(X,F)$$
is an isomorphism.
\end{prop}

\begin{proof}
See \cite[Theorem 3.4.17]{kelly}.
\end{proof}

\begin{cor}\label{cor:ldhdescenttransfers}
Assume that $X$ is of finite dimension, and
let $C$ be a complex of $\mathbf{Z}_{(\ell)}$-linear
$\cdh$-sheaves with transfers over $X$.
Then the comparison map of hypercohomologies
$$H^n_\cdh(X,C)\to H^n_\ldh(X,C)$$
is an isomorphism for all $n$.
\end{cor}

\begin{proof}
Note that, for $t=\cdh$ or $t=\ldh$, the forgetful functor
from $\mathbf{Z}_{(\ell)}$-linear $t$-sheaves with tranfers
to $\mathbf{Z}_{(\ell)}$-linear $t$-sheaves on the big site of $X$
is exact (this follows from the stronger results given by
\cite[Prop. 3.4.15 and 3.4.16]{kelly} for instance).
Therefore, we have a canonical spectral sequence of the form
$$E^{p,q}_2=H^p_t(X,H^q(C)_t)\Rightarrow H^{p+q}_t(X,C)\, .$$
As the cohomological dimension with respect to the $\cdh$-topology
is bounded by the dimension, this spectral sequence strongly
converges for $t=\cdh$.
Proposition \ref{prop:ldhdescenttransfers}
thus implies that,
for $t=\ldh$, the groups $E^{p,q}_2$ vanish for $p<0$ or $p>\mathrm{dim}\, X$,
so that this spectral sequence also converges in this case. Therefore,
as these two spectral sequences agree on the $E_2$ term, we conclude that
they induce an isomorphism on $E_\infty$.
\end{proof}

\begin{cor}\label{cor:ldhdescentDMcdh}
For $X$ of finite dimension and $R$ an $\mathbf{Z}_{(\ell)}$-algebra,
any object of the triangulated category $\uDM_\cdh(X,R)$ satisfies $\ldh$-descent (see \cite[Definition 3.2.5]{CD3}).
\end{cor}

\begin{lm}\label{lm:ldhdescentHZmodfinitetype}
Assume that $X$ is of finite type over the perfect field $k$.
Consider a prime $\ell$ which is distinct from the characteristic exponent
of $k$. If $R$ is a $\mathbf{Z}_{(\ell)}$-algebra, then any compact
object of $\HZ_{X/k}\Mod$ satisfies $\ldh$-descent.
\end{lm}

\begin{proof}
As $X$ is allowed to vary, it is sufficient to prove that, for any
constructible $\HZ_{X/k}$-modules $M$ and any $\ldh$-hypercover $p_\bullet :U_\bullet\to X$,
the map
\begin{equation}\label{proofoflm:ldhdescentHZmodfinitetype1}
\derR\Gamma(X,M)\to\derR\varprojlim_{\Delta_n}\derR\Gamma(U_n,p^*_nM)
\end{equation}
is an isomorphism.
The category of compact objects of $\HZ_X\Mod$
is the thick subcategory generated by objects of the form
$\derR f_*\HZ_{Y/k}(p)$ for $f:Y\to X$ a projective map and $p$ an integer
(this follows right away from the fact that te analogous property is true
in $\SH$).
We may thus assume that $M=\derR f_*\HZ_{Y/k}(p)$.
We can then form the following pullback in the
category of simplicial schemes.
$$\xymatrix{
V_\bullet\ar[r]^g\ar[d]_{q_\bullet}&U_\bullet\ar[d]^{p_\bullet}\\
Y\ar[r]^f&X
}$$
Using the proper base change formula for $\HZ_{(-)/k}$-modules,
we see that the map \eqref{proofoflm:ldhdescentHZmodfinitetype1} is isomorphic
to the map
\begin{equation}\label{proofoflm:ldhdescentHZmodfinitetype2}
\derR\Gamma(Y,\HZ_{Y/k}(p))\to\derR\varprojlim_{\Delta_n}\derR\Gamma(V_n,\HZ_{V_n/k}(p))\, .
\end{equation}
By virtue of Kelly's $\ldh$-descent theorem \cite[Theorem 5.3.7]{kelly},
the map \eqref{proofoflm:ldhdescentHZmodfinitetype2} is an isomorphism.
\end{proof}

\begin{lm}\label{lm:ldhdescentreg}
Let $X$ be a $k$-scheme of finite type. Assume that $R$ is
a $\mathbf{Z}_{(\ell)}$-algebra for $\ell$ a prime number distinct from
the characteristic exponent of $k$.
Let $M$ be an object of $\uDM(X,R)$ satisfying $\ldh$-descent
on the site of smooth $k$-schemes over $X$:
for any $X$-scheme of finite type $Y$ which is smooth
over $k$ and any $\ldh$-hypercover $p:U_\bullet\to Y$
such that $U_n$ is smooth over $k$ for any $n\geq 0$, the map
$$\derR\Hom_{\uDM(X,R)}(R(Y),M(p))
\to\derR\varprojlim_{\Delta_n}\derR\Hom_{\uDM(X,R)}(R(U_n),M(p))$$
is an isomorphism in the derived category of $R$-modules.
Then, for any $X$-scheme $Y$ which is smooth over $k$ and any integer $p$,
the canonical map
$$\derR\Hom_{\uDM(X,R)}(R(Y),M(p))\to
\derR\Hom_{\uDM_\cdh(X,R)}(R(Y),M_\cdh(p))$$
is an isomorphism.
\end{lm}

\begin{proof}
Let us denote by $R\{1\}$ the complex
$$R\{1\}=R(1)[1]=\mathrm{ker}(R(\mathbf{A}^1_X-\{0\})\to R)$$ induced by the
structural map $\mathbf{A}^1-\{0\}\times X\to X$.
We may consider that the object $M$ is a fibrant $R\{1\}$-spectrum
in the category of complexes of $R$-linear sheaves with transfers
on the category of $X$-schemes of finite type.
In particular, $M$ corresponds to a collection of
complexes of $R$-linear sheaves with transfers $(M_n)_{\geq 0}$
together with maps $R\{1\}\otimes_R M_n\to M_{n+1}$ such that
we have the following properties.
\begin{itemize}
\item[(i)] For any integer $n\geq 0$ and any $X$-scheme of finite type $Y$, the map
$$\Gamma(Y,M_n)\to\derR\Gamma(Y,M_n)$$ is an isomorphism in the derived
category of $R$-modules (where $\derR\Gamma$ stands for the derived
global section with respect to the Nisnevich topology).
\item[(ii)] For any integer $n\geq 0$, the map
$$M_n\to\derR\uHom(R\{1\},M_{n+1})$$
is an isomorphism in the derived category of Nisnevich sheaves with tranfers
(where $\derR\uHom$ stands for the derived internal Hom).
\end{itemize}
We can choose another $R\{1\}$-spectum $N=(N_n)_{n\geq 0}$ of $\cdh$-sheaves
with transfers, together
with a cofibration of spectra $M\to N$ such that $M_n\to N_n$ is
a quasi-isomor\-phism locally for the $\cdh$-topology, and such that each $N_n$
satisfies $\cdh$-descent: we do this by induction as follows. First, $N_0$
is any fibrant resolution of $(M_0)_\cdh$ for the $\cdh$-local model structure
on the category of complexes of $\cdh$-sheaves with transfers.
If $N_n$ is already constructed, we denote by $E$ the pushout of $M_n$
along the map $R\{1\}\otimes_R M_n\to R\{1\}\otimes_R N_n$, and we factor
the map $E_\cdh\to 0$ into a trivial cofibration followed by a fibration
in the $\cdh$-local model structure.

Note that, for any $X$-scheme $Y$ which is smooth over $k$, the map
$$H^i(Y,M_n)\to H^i(Y,N_n)$$
is an isomorphism of $R$-modules for any integers $i\in\mathbf{Z}$ and $n\geq 0$.
Indeed, as, by virtue of Gabber's theorem of resolution of singularities
by $\ldh$-alterations \cite[Exp.~IX, Th.~1.1]{gabber},
one can write both sides with the Verdier formula in the
following way (because of our hypothesis on $M$ and by construction of $N$):
$$H^i(Y,E)\simeq
\varinjlim_{U_\bullet\to Y}H^i(\derR\varprojlim_{\Delta_j}\Gamma(U_j,M_n))
\ \text{for $E=M_n$ or $E=N_n$,}$$
where $U_\bullet\to Y$ runs over the filtering category of $\ldh$-hypercovers
of $Y$ such that each $U_j$ is smooth over $k$.
It is also easy to see from this formula that each $N_n$ is
$\mathbf{A}^1$-homotopy invariant and that the maps
$$N_n\to\uHom(R\{1\},N_{n+1})$$
are isomorphisms. In other words, $N$ satisfies the analogs of
properties (i) and (ii) above with respect to the $\cdh$-topology.
We thus get the following identifications for $p\geq 0$:
\begin{align*}
\Gamma(Y,M_p)&=\derR\Hom_{\uDM(X,R)}(R(Y),M(p))\\
\Gamma(Y,N_p)&=\derR\Hom_{\uDM_\cdh(X,R)}(R(Y),M_\cdh(p))\, .
\end{align*}
The case where $p<0$ follows from the fact that, for $d=-p$, $R(Y)(d)[2d]$
is then a direct factor of $R(Y\times\mathbf{P}^{d})$ (by the projective
bundle formula in $\uDM_{\cdh}(X,R)$).
\end{proof}

\begin{lm}\label{lm:ldhdescentreg2}
Let $X$ be a smooth separated
$k$-scheme of finite type. Assume that $R$ is
a $\mathbf{Z}_{(\ell)}$-algebra for $\ell$ a prime number distinct from
the characteristic exponent of $k$.
If $M$ and $N$ are two constructible objects of $\DM(X,R)$,
then the comparison map
$$\derR\Hom_{\DM(X,R)}(M,N)\to\derR\Hom_{\DM_\cdh(X,R)}(M,N)$$
is an isomorphism in the derived category of $R$-modules.
\end{lm}

\begin{proof}
It is sufficient to prove this in the case where $M=R(Y)(p)$
for $Y$ a smooth $X$-scheme and $p$ any integer.
By virtue of Lemma \ref{lm:ldhdescentreg},
it is sufficient to prove that any constructible object of $\DM(X,R)$
satisfies $\ldh$-descent on the site of $X$-schemes which are smooth over $k$.
By virtue of Theorem \ref{thm:comparisonnis},
it is thus sufficient to prove the analogous property for constructible
$\HZ_X$-modules, which follows from Lemma \ref{lm:ldhdescentHZmodfinitetype}.
\end{proof}

\begin{proof}[Proof of Theorem \ref{thm:comparisoncdh}]
It is sufficient to prove that the restriction of the comparison functor
\begin{equation}\label{eq:proofcomparisoncdh1}
\HZ_{X/k}\Mod\to\DMcdh(X,R) \ , \quad M\mapsto \tau^*(M)
\end{equation}
to constructible $\HZ_{X/k}$-modules is fully faithful (by virtue
of \cite[Corollary 1.3.21]{CD3}, this is because both
triangulated categories are compactly generated and because
the functor \eqref{eq:proofcomparisoncdh1} preserves the canonical
compact generators). It is easy to see that this functor is fully faithful
(on constructible objects) if and only if, for any prime $\ell\neq p$,
its $R\otimes\mathbf{Z}_{(\ell)}$-linear version has this property
(this is because the functor \eqref{eq:proofcomparisoncdh1}
preserves compact objects, which implies that its right adjoint
commutes with small sums, hence both functors
commute with the operation of tensoring by $\ZZ_{(\ell)}$).
Therefore, we may assume that a prime number $\ell\neq p$ is given and that
$R$ is a $\mathbf{Z}_{(\ell)}$-algebra.
We will then prove the property of being fully faithful
first in the case where $X$ is
of finite type over $k$, and then, by a limit argument, in general.

Assume that $X$ is of finite type over $k$,
and consider constructible $\HZ_{X/k}$-modules $M$ and $N$.
We want to prove that, the map
\begin{equation}\label{eq:proofcomparisoncdh1bis}
\derR\Hom_{\HZ_{X/k}\Mod}(M,N)\to\derR\Hom_{\DM_{\cdh}(X,R)}(\tau^*(M),\tau^*(N))
\end{equation}
is an isomorphism (here all the $\derR\Hom$'s take their values in
the triangulated category of topological $S^1$-spectra;
see \cite[Theorem 3.2.15]{CD3} for the existence (and uniqueness) of such an enrichment).
By virtue of Gabber's theorem of resolution of singularities
by $\ldh$-alterations \cite[Exp.~IX, Th.~1.1]{gabber},
we can choose an $\ldh$-hypercover $p_\bullet:U_\bullet\to X$,
with $U_n$ smooth, separated, and of finite type over $k$ for any non negative
integer $n$. We then have the following chain of isomorphisms,
justified respectively by $\ldh$-descent for constructible
$\HZ_{X/k}$-modules (Lemma \ref{lm:ldhdescentHZmodfinitetype}),
by the comparison theorem relating the category
of $\HZ$-modules with $\DM$ over regular $k$-schemes (Theorem \ref{thm:comparisonnis}),
by Lemma \ref{lm:ldhdescentreg2}, and finally by the fact that
any complex of $R$-modules with transfers on the category of separated
$X$-schemes of finite type which satisfies $\cdh$-descent must
satisfy $\ldh$-descent as well (Corollary \ref{cor:ldhdescenttransfers}):
\begin{align*}
\derR\Hom_{\HZ_{X/k}\Mod}(M,N)
&\simeq\derR\varprojlim_{\Delta_n}
\derR\Hom_{\HZ_{U_n}\Mod}(\derL p^*_nM,\derL p^*_nN)\\
&\simeq\derR\varprojlim_{\Delta_n}
\derR\Hom_{\DM(U_n,R)}(\derL p^*_n\, t^*(M),\derL p^*_n\, t^*(N))\\
&\simeq\derR\varprojlim_{\Delta_n}
\derR\Hom_{\DM_\cdh(U_n,R)}(\derL p^*_n\, \tau^*(M),\derL p^*_n\, \tau^*(N))\\
&\simeq\derR\Hom_{\DM_{\cdh}(X,R)}(\tau^*(M),\tau^*(N))\, .
\end{align*}

It remains to treat the case of an arbitrary noetherian $k$-scheme $X$.
It is easy to see that the property that the functor \eqref{eq:proofcomparisoncdh1}
is fully faithful (on constructible objects) is local on $X$ with respect
to the Zariski topology. Therefore, we may assume that $X$ is affine with
structural ring $A$. We can then write $A$ as a filtering colimit of
$k$-algebras of finite type $A_i\subset A$, so that we obtain a projective
system of $k$-schemes of finite type $\{X_i=\Spec A_i\}_i$ with
affine and dominant transition maps, such that $X=\varprojlim_i X_i$.
But then, by continuity (applying Proposition \ref{prop:continuity&2lim}
twice, using Lemma \ref{lm:trivial_continuity}
for $\HZ_{X/k}\Mod$, and Example \ref{ex:continuity}(2) for $\DMcdh(X,R)$),
we have cano\-ni\-cal equivalences of categories at the level
of constructible objects:
\begin{align*}
\HZ_{X/k}\Mod_c
&\simeq 2\text{-}\varinjlim_i\HZ_{X_i/k}\Mod_c\\
&\simeq 2\text{-}\varinjlim_i\DMcdh(X_i,R)_c\\
&\simeq\DMcdh(X,R)_c\, .
\end{align*}
In particular, the functor \eqref{eq:proofcomparisoncdh1} is
fully faithful on constructible objects, and this ends the proof.
\end{proof}

\begin{cor}\label{cor:DMcdhreg}
Let $X$ be a regular noetherian $k$-scheme of finite dimension.
Then the canonical functor
$$\DM(X,R)\to\DM_\cdh(X,R)$$
is an equivalence of symmetric monoidal triangulated categories. 
\end{cor}

\begin{proof}
This is a combination of
Theorems \ref{thm:comparisonnis} and \ref{thm:comparisoncdh},
and of Proposition \ref{prop:HZmodrelatif1}.
\end{proof}

Remark that we get for free the following result, which generalizes
Kelly's $\ldh$-descent theorem:

\begin{thm}\label{thm:ldhdescentHZmod}
Let $k$ be a field of characteristic exponent $p$, $\ell$
a prime number distinct from $p$, and $R$ a $\mathbf{Z}_{(\ell)}$-algebra.
Then, for any noetherian $k$-scheme of finite dimension $X$, any object of $\HZ_{X/k}\Mod$
satisfies $\ldh$-descent.
\end{thm}

\begin{proof}
This follows immediately from Theorem \ref{thm:comparisoncdh}
and from Corollary \ref{cor:ldhdescentDMcdh}.
\end{proof}

Similarly, we see that $\DM_\cdh$ is continuous is
a rather general sense.

\begin{thm}\label{thm:continuityabsDMcdh}
The motivic category $\DM_\cdh(-,R)$ has the properties of
localization with respect to any closed immersion
as well as the property of
continuity with respect to arbitrary projective systems
with affine transition maps
over the category of noetherian $k$-schemes of finite dimension.
\end{thm}

\begin{proof}
Since $\HZ_{(-)/k}\Mod$ has these properties, Theorem \ref{thm:comparisoncdh}
allows to transfer it to $\DM_\cdh(-,R)$.
\end{proof}

%

\section{Finiteness}

\begin{num}\label{num:constructpremotivic}
In this section, all the functors are derived
functors, but we will drop $\derL$ or $\derR$ from the notations.
The triangulated motivic category $\DM_\cdh(-,R)$ is endowed
with the six operations $\otimes_R$, $\uHom_R$, $f^*$, $f_*$, $f_!$ and $f^!$
which satisfy the usual properties; see \cite[Theorem 2.4.50]{CD3}
for a summary.

Recall that an object of $\DM_\cdh(X,R)$
is constructible if and only if it is compact.
Here is the behaviour of the six operations with respect to
constructible objects in $\DM_\cdh(-,R)$, when we restrict
ourselves to $k$-schemes (see \cite[4.2.5, 4.2.6, 4.2.10, 4.2.12]{CD3}):
\begin{itemize}
\item[(i)] constructible objects are stable by tensor products;
\item[(ii)] for any morphism $f:X\to Y$, the functor
$f^*:\DM_\cdh(Y,R)\to\DM_\cdh(X,R)$ preserves constructible objects;
\item[(iii)] The property of being constructible is local for the
Zariski topology;
\item[(iii)] given a closed immersion $i:Z\to X$ with open
complement $j:U\to X$, an object $M$ of $\DM_\cdh(X,R)$
is constructible if and only if $i^*(M)$ and $j^*(M)$
are constructible;
\item[(iv)] the functor $f_!:\DM_\cdh(X,R)\to\DM_\cdh(Y,R)$
preserves constructible objects for any separated morphism of finite type $f:X\to Y$.
\end{itemize}
\end{num}

\begin{prop}\label{prop:abspurity}
Let $i:Z\to X$ be a closed immersion of codimension $c$
between regular $k$-schemes.
Then there is a canonical isomorphism $i^!(R_X)\simeq R_Z(-c)[-2c]$
in $\DM_\cdh(Z,R)$.
\end{prop}

\begin{proof}
In the case where $X$ and $Z$ are smooth over $k$, this is
a direct consequence of the relative purity theorem.
For the general case, using the reformulation of the absolute purity theorem
of \cite[Appendix, Theorem A.2.8(ii)]{CD4}, we see that
it is sufficient to prove this proposition locally
for the Zariski topology over $X$. Therefore we may assume
that $X$ is affine. Since $\DM_\cdh(-,R)$ is continuous (\ref{thm:continuityabsDMcdh}),
using Popescu's theorem and \cite[4.3.12]{CD3}, we
see that it is sufficient to treat the case where $X$ is smooth of finite
type over $k$. But then, this is
a direct consequence of the relative purity theorem.
\end{proof}

\begin{prop}\label{prop:remindertrace}
Let $f:X\to Y$ be a morphism of noetherian $k$-schemes.
Assume that both $X$ and $Y$ are integral and that $f$ is
finite and flat of degree $d$.
Then, there is a canonical natural transformation
$$\mathit{Tr}_f:\derR f_*\derL f^*(M)\to M$$
for any object $M$ of $\DM_\cdh(X,R)$
such that the composition with the unit of the adjunction
$(\derL f^*,\derR f_*)$
$$M\to \derR f_*\derL f^*(M)\xrightarrow{\mathit{Tr}_f} M$$
is $d$ times the identity of $M$.
\end{prop}

\begin{proof}
As in paragraphs \ref{num:elementary_traces}
and \ref{num:elementary_traces2} (simply replacing
$\uDM(X,R)$ and $\DM(X,R)$ by $\uDM_\cdh(X,R)$
and $\DM_\cdh(X,R)$, respectively), we construct
$$\mathit{Tr}_f:\derR f_*(R_X)=\derR f_*\derL f^*(R_Y)\to R_Y$$
such that the composition with the unit
$$R\to \derR f_*(R_X)\xrightarrow{\mathit{Tr}_f} R_Y$$
is $d$. Then, since $f$ is proper, we have a projection
formula
$$\derR f_*(R_X)\otimes^\derL_R M\simeq\derR f_*\derL f^*(M)$$
and we construct
$$\mathit{Tr}_f:\derR f_*\derL f^*(M)\to M$$
as
$$M\otimes^\derL_R\big(\derR f_*(R_X)\xrightarrow{\mathit{Tr}_f} R_Y \big)\, .$$
This ends the construction of $\mathit{Tr}_f$ and the proof of this
proposition.
\end{proof}

\begin{thm}\label{thm:constructmotivic}
The six operations preserve constructible objects in $\DM_\cdh(-,R)$
over quasi-excellent $k$-schemes.
In particular, we have the following properties.
\begin{itemize}
\item[(a)] For any morphism of finite type between quasi-excellent
$k$-schemes, the functor $f_*:\DM_\cdh(X,R)\to\DM_\cdh(Y,R)$
preserves constructible objects.
\item[(b)] For any separated morphism of finite type between quasi-excellent
$k$-schemes $f:X\to Y$, the functor
$f^!:\DM_\cdh(Y,R)\to\DM_\cdh(X,R)$ preserves constructible objects.
\item[(c)] If $X$ is a quasi-excellent $k$-scheme, for any constructible
objects $M$ and $N$ of $\DM_\cdh(M,N)$, the object $\uHom_R(M,N)$
is constructible.
\end{itemize}
\end{thm}

\begin{proof}[Sketch of proof]
It is standard that properties (b) and (c) are
corollaries of property (a); see the proof of
\cite[Cor. 6.2.14]{CD4}, for instance.
Also, to prove (a), the usual argument (namely
\cite[Lem. 2.2.23]{ayoub}) shows that it is sufficient
to prove that, for any morphism of finite type $f:X\to Y$,
the object $f_*(R_X)$ is constructible.
As one can work locally for the Zariski topology on $X$
and on $Y$, one may assume that $f$ is separated (e.g. affine)
and thus that $f=pj$ with $j$ an open immersion and $p$ a proper
morphism. As $p_!=p_*$ is already known to preserve
constructible objects, we are thus reduced to prove that,
for any dense open immersion $j:U\to X$,
the object $j_*(R_U)$ is constuctible. This is where the serious
work begins. First, using the fact that constructible objects
are compact, for any prime $\ell\neq p$, the triangulated
category $\DM_\cdh(X,R\otimes\ZZ_{(\ell)})$ is the
idempotent completion of the triangulated category
$\DM_\cdh(X,R)\otimes\ZZ_{(\ell)}$.
Therefore, using \cite[Appendix, Prop. B.1.7]{CD4},
we easily see that it is sufficient to consider the case where $R$
is a $\ZZ_{(\ell)}$-algebra for some prime $\ell\neq p$.
The rest of the proof consists to follow word for word a beautiful
argument of Gabber: the very proof of \cite[Lem. 6.2.7]{CD4}.
Indeed, the only part of the proof of \emph{loc. cit.} which is not
meaningful in an abstract motivic triangulated category is
the proof of the sublemma \cite[6.2.12]{CD4}, where we need
the existence of trace maps for flat finite surjective morphisms
satisfying the usual degree formula. In the case of $\DM_\cdh(X,R)$,
we have such trace maps
natively: see Proposition \ref{prop:remindertrace}.
\end{proof}

\section{Duality}

In this section, we will consider a field $K$ of
exponential characteristic $p$, and will focus our attention
on $K$-schemes of finite type. As anywhere else in this
article, the ring of coefficients $R$ is assumed to be a $\ZZ[1/p]$-algebra.

\begin{prop}\label{prop:purelyinsepext}
Let $f:X\to Y$ be a surjective finite radicial morphism of
noetherian $K$-schemes of finite dimension. Then
the functor
$$\derL f^*:\DM_\cdh(Y,R)\to\DM_\cdh(X,R)$$
is an equivalence of categories and is
canonically isomorphic to the functor $f^!$.
\end{prop}

\begin{proof}
By virtue of \cite[Prop.~2.1.9]{CD3}, it is sufficient to prove that
pulling back along such a morphism $f$ induces
a conservative functor $\derL f^*$ (the fact that $\derL f^*\simeq f^!$
come from the fact that if $\derL f^*$ is an equivalence of categories, then
so is its right adjoint $f_!\simeq\derR f_*$, so that $\derL f^*$ and $f^!$
must be quasi-inverses of the same equivalence of categories).
Using the localization property as well as a suitable noetherian
induction, it is sufficient to check this property generically
on $Y$. In particular, we may assume that $Y$ and $X$ are integral
and that $f$ is moreover flat. Then the degree of $f$ must be some
power of $p$, and Proposition \ref{prop:remindertrace} then
implies that the functor $\derL f^*$ is faithful (and thus conservative).
\end{proof}

\begin{prop}\label{prop:generators}
Let $X$ be a scheme of finite type over $K$, and $Z$
a fixed nowhere dense closed subscheme of $X$.
Then the category of constructible motives $\DM_{\cdh,c}(X,R)$
is the smallest thick subcategory containing objects of the
form $f_!(R_Y)(n)$, where $f:Y\to X$ is a projective morphism
with $Y$ regular, such that $f^{-1}(Z)$ is either empty, the whole scheme $Y$
itself, or the support of a strict normal crossing divisor, while $n$ is
any integer.
\end{prop}

\begin{proof}
Let $\mathcal{G}$ be the family of objects of the form
$f_!(R_Y)(n)$, with $f:Y\to X$ a projective morphism,
$Y$ regular, $f^{-1}(Z)$ either empty or
the support of a strict normal crossing divisor, and $n$
any integer. We already know that any element of $\mathcal{G}$
is constructible. Since the constructible objects of $\DM_\cdh(X,R)$
precisely are the compact objects, which do form a generating
family of the triangulated category $\DM_\cdh(X,R)$, it is sufficient
to prove that the family $\mathcal{G}$ is generating.
Let $M$ be an object of $\DM_\cdh(X,R)$ such that
$\Hom(C,M[i])=0$ for any element $C$ of $\mathcal{G}$ and any integer $i$.
We want to prove that $M=0$.
For this, it is sufficient to prove that $M\otimes\ZZ_{(\ell)}=0$
for any prime $\ell$ which not invertible in $R$ (hence, in particular,
is prime to $p$). Since, for any compact object $C$ of $\DM_\cdh(X,R)$,
we have
$$\Hom(C,M[i])\otimes\ZZ_{(\ell)}\simeq\Hom(C,M\otimes\ZZ_{(\ell)}[i])\, ,$$
and since $f_!$ commutes with tensoring with $\ZZ_{(\ell)}$
(because it commutes with small sums), we may assume that $R$ is
a $\ZZ_{(\ell)}$-algebra for some prime number $\ell\neq p$.
Under this extra hypothesis, we will prove directly that
$\mathcal{G}$ generates the thick category of compact objects.
Let $T$ be the smallest thick subcategory of $\DM_\cdh(X,R)$
which contains the elements of $\mathcal{G}$. 

For $Y$ a separated $X$-scheme of finite type, we put
$$M^{\mathit{BM}}(Y/X)=f_!(R_Y)$$
with $f:Y\to X$ the structural morphism.
If $Z$ is a closed subscheme of $Y$ with open complement $U$,
we have a canonical distinguished triangle
$$M^{\mathit{BM}}(U/X)\to M^{\mathit{BM}}(Y/X)\to M^{\mathit{BM}}(Z/X)\to M^{\mathit{BM}}(Z/X)[1]\, .$$
We know that the subcategory of constructible objects
of $\DM_\cdh(X,R)$ is the smallest thick subcategory which
contains the objects of the form $M^{\mathit{BM}}(Y/X)(n)$ for $Y\to X$
projective, and $n\in\ZZ$; see \cite[Lem.~2.2.23]{ayoub}.
By $\cdh$-descent (as formulated in \cite[Prop.~3.3.10~(i)]{CD3}),
we easily see that objects of the form $M^{\mathit{BM}}(Y/X)(n)$ for $Y\to X$
projective, $Y$ integral, and $n\in\ZZ$, generate the thick
subcategory of constructible objects of $\DM_\cdh(X,R)$.
By noetherian induction on the dimension of such
a $Y$, it is sufficient to prove
that, for any projective $X$-scheme $Y$, there exists a dense open
subscheme $U$ in $Y$ such that $M^{\mathit{BM}}(U/X)$ belongs to $T$.
By virtue of Gabber's refinement of de~Jong's theorem of resolution
of singularities by alterations \cite[Exp.~X, Theorem 2.1]{gabber},
there exists a projective
morphism $Y'\to Y$ which is generically flat, finite surjective
of degree prime to $\ell$, such that $Y'$ is regular, and such that
the inverse image of $Z$ in $Y'$ is either empty, the whole scheme $Y'$,
or the support of a strict normal crossing divisor. Thus, by induction, for any
dense open subscheme $V\subset Y'$, the motive $M^{\mathit{BM}}(V/X)$ belongs
to $T$. But, by assumption on $Y'\to Y$, there exists a dense
open subscheme $U$ of $Y$ such that, if $V$ denote the pullback of $U$
in $Y'$, the induced map $V\to U$
is a finite, flat and surjective morphism between integral $K$-schemes
and is of degree prime to $\ell$. By virtue of Proposition \ref{prop:remindertrace},
the motive $M^{\mathit{BM}}(U/X)$ is thus a direct factor of $M^{\mathit{BM}}(V/X)$, and since
the latter belongs to $T$, this shows that $M^{\mathit{BM}}(Y/X)$
belongs to $T$ as well, and this achieves the proof.
\end{proof}

\begin{thm}\label{thm:duality}
Let $X$ be a separated $K$-scheme of finite type, with structural
morphism $f:X\to\Spec(K)$. Then the object $f^!(R)$ is dualizing.
In other words, for any constructible object $M$ in $\DM_\cdh(X,R)$, the
natural map
\begin{equation}\label{eq:thm:duality1}
M\to\derR\uHom_R(\derR\uHom_R(M,f^!(R),f^!(R)))
\end{equation}
is an isomorphism.
In particular, the natural map
\begin{equation}\label{eq:thm:duality2}
R_X\to\derR\uHom_R(f^!(R),f^!(R))
\end{equation}
is an isomorphism in $\DM_\cdh(X,R)$.
\end{thm}

\begin{proof}
By virtue of Proposition \ref{prop:generators},
it is sufficient to prove that the map \eqref{eq:thm:duality1}
is an isomorphism for $M=p_!(R_Y)$ with $p:Y\to X$ projective and $Y$
regular. We then have
$$\derR\uHom_R(M,f^!(R))\simeq p_!\derR\uHom_R(R_Y,p^!f^!(R))\simeq p_!p^!(f^!(R))\, ,$$
hence
\begin{align*}
\derR\uHom_R(\derR\uHom_R(M,f^!(R),f^!(R))
&\simeq\derR\uHom_R(p_!p^!(f^!(R)),f^!(R))\\
&\simeq p_!\derR\uHom_R(p^!f^!(R),p^!f^!(R))\, .
\end{align*}
The map \eqref{eq:thm:duality1} is thus, in this case, the image
by the functor $p_!$ of the map $R_Y\to \derR\uHom_R(p^!f^!(R),p^!f^!(R))$.
In other words, it is sufficient to prove that the map \eqref{eq:thm:duality2}
is an isomorphism in the case where $X$ is regular (and projective over $K$).
But $X$ is then smooth on a finite purely inseparable extension $L$ of $K$.
By virtue of Proposition \ref{prop:purelyinsepext}, we may assume that $X$
is actually smooth over $K$. But then, if $d$ is the dimension of $X$, since
$\DM_\cdh$ is oriented, we have a purity isomorphism $f^!(R)\simeq R_X(d)(2d]$.
Since we obviously have the identification,
$R_X\simeq\derR\uHom_R(R_X(d),R_X(d))$, this achieves
the proof.
\end{proof}

\begin{rem}
The preceding theorem means that, if we restrict to separated $K$-schemes
of finite type, the whole formalism of Grothendieck-Verdier duality
holds in the setting of $R$-linear $\cdh$-motives.
In other words, for a separated $K$-scheme of finite type $X$
with structural map $f:X\to\Spec(K)$, we define the functor $D_X$ by
$$D_X(M)=\derR\uHom_R(M,f^!(R))$$
for any object $M$ of $\DM_\cdh(X,R)$.
We already know that $D_X$ preserves constructible objects and that
the natural map $M\to D_X(D_X(M))$ is invertible for any constructible object
$M$ of $\DM_\cdh(X,R)$. For any objects $M$ and $N$ of $\DM_\cdh(X,R)$, if
$N$ is constructible, we have a natural isomorphism
\begin{equation}\label{eq:dualityfexcept0}
\derR\uHom_R(M,N)\simeq D_X(M\otimes^\derL_RD_X(N))\, .
\end{equation}
For any $K$-morphism between separated $K$-schemes of finite type
$f:Y\to X$, and for any constructible objects $M$ and $N$ in
$\DM_\cdh(X,R)$ and $\DM_\cdh(Y,R)$, respectively, we have the
following natural identifications.
\begin{align}
D_Y(f^*(M))&\simeq f^!(D_X(M))\label{eq:dualityfexcept1}\\
f^*(D_X(M))&\simeq D_Y(f^!(M))\label{eq:dualityfexcept2}\\
D_X(f_!(N))&\simeq f_*(D_Y(N))\label{eq:dualityfexcept3}\\
f_!(D_Y(N))&\simeq D_X(f_*(N))\label{eq:dualityfexcept4}
\end{align}
\end{rem}

\section{Bivariant cycle cohomology}

\begin{prop}\label{prop:insepclosureeff}
Let $K$ be a field of characteristic exponent $p$, and
$K^s$ its inseparable closure.
\begin{itemize}
\item[(a)] The map
$u:\Spec(K^s)\to\Spec(K)$ induces fully faithful functors
$$u^*:\uDMe(K,R)\to\uDMe(K^s,R)
\quad\text{and}\quad
u^*:\uDMe_\cdh(K,R)\to\uDMe_\cdh(K^s,R)\, .$$
\item[(b)] We have a canonical equivalence of categories
$$\DMe(K^s,R)\simeq\uDMe_\cdh(K^s,R)\, .$$
\item[(c)] At the level of non-effective motives, we have
canonical equivalences of categories
$$\DM(K,R)\simeq\DM_\cdh(K,R)\simeq\uDM_\cdh(K,R)\, .$$
\item[(d)] The pullback functor
$$u^*:\DM(K,R)\to\DM(K^s,R)$$
is an equivalence of categories.
\end{itemize}
\end{prop}

\begin{proof}
In all cases, $u^*$ has a right adjoint $\derR u_*$
which preserves small sums (because $u^*$ preserves
compact objects, which are generators).

Let us prove that the functor
$$u^*:\uDMe(K)\to\uDMe(K^s)$$
is fully faithful.
By continuity (see \cite[Example 11.1.25]{CD3}),
it is sufficient to prove that, for any finite
purely inseparable extension $L/K$, the pullback functor
along the map $v:\Spec(L)\to\Spec(K)$,
$$v^*:\DMe(K,R)\to\DMe(L,R)\, ,$$
is fully faithful.
As, for any field $E$, we have a fully faithful
embedding
$$\DMe(E,R)\to\uDMe(E,R)$$
which is compatible with pulbacks (see \cite[Prop.~11.1.19]{CD3}),
it is sufficient to prove that the pullback functor
$$v^*:\uDMe(K,R)\to\uDMe(L,R)$$
is fully faithful. In this case, the functor $v^*$ has a left adjoint $v_\sharp$,
and we must prove that the co-unit
$$v_\sharp\, v^*(M)\to M$$
is fully faithful for any object $M$ of $\uDMe(K)$.
The projection formula $v_\sharp\, v^*(M)=v_\sharp(R)\otimes^\derL_R M$
reduces to prove that the co-unit $v_\sharp\, v^*(R)\to R$ is
an isomorphism, which follows right away from \cite[Prop.~9.1.14]{CD3}.
The same arguments show that the functor
$$u^*:\uDMe_\cdh(K,R)\to\uDMe_\cdh(K^s,R)$$
is fully faithful.

The canonical functor 
$$\DMe(L,R)\to\uDMe_\cdh(L,R)$$
is an equivalence of categories for any perfect
field $L$ of exponent characteristic $p$ by a result
in Kelly's thesis (more precisely the right
adjoint of this functor is an equivalence
of categories; see the last assertion of \cite[Cor.~5.3.9]{kelly}).

The fact that the functor
$$u^*:\DM_c(K,R)\to\DM_c(K^s,R)$$
is an equivalence of categories follows by continuity from
the fact that the pullback functor
$$\DM_c(K,R)\to\DM_c(L,R)$$
is an equivalence of categories for any finite
purely inseparable extension $L/K$
(see \cite[Prop.~2.1.9 and 2.3.9]{CD3}).
As the right adjoint of $u^*$ preserves small sums,
this implies that $u^*:\DM(K,R)\to\DM(K^s,R)$ is fully faithful.
Since any compact object of $\DM(K^s,R)$ is in the essential
image and since $\DM(K^s,R)$ is compactly generated, this
proves that $u^*:\DM(K,R)\to\DM(K^s,R)$ is an equivalence
of categories; see \cite[Corollary 1.3.21]{CD3}.

As we already know that the functor
$$\DM(K,R)\to\DM_\cdh(K,R)$$
is an equivalence of categories (Cor.~\ref{cor:DMcdhreg}), it remains to prove
that the functor
$$\DM_\cdh(K,R)\to\uDM_\cdh(K,R)$$
is an equivalence of categories (or even an equality).
Note that we have
$$\DM_\cdh(L,R)=\uDM_\cdh(L,R)$$
for any perfect field of exponent characteristic $p$.
This simply means that motives of the form $M(X)(n)$,
for $X$ smooth over $L$ and $n\in\ZZ$, do form a generating
family of $\uDM(L,R)$. To prove this, let us consider an 
object $C$ of $\DM_\cdh(L,R)$ such that 
$$\Hom(M(X)(n),C[i])=0$$
for any smooth $L$-scheme $X$ and any integers $n$ and $i$.
To prove that $C=0$, since, for any compact object $E$
and any localization $A$ of the
ring $\ZZ$, the functor $\Hom(E,-)$
commutes with tensoring by $A$, we may assume that $R$ is a $\ZZ_{(\ell)}$-algebra
for some prime number $\ell\neq p$.
Under this extra assumption, we know that the object $C$
satisfies $\ldh$-descent (see Corollary \ref{cor:ldhdescentDMcdh}).
Since, by Gabber's theorem, any scheme of finite type over $L$ is smooth
locally for the $\ldh$-topology, this proves that $C=0$.

Finally, let us consider an object $C$ of $\uDM_\cdh(K,R)$
such that $\Hom(M,C)=0$ for any object $M$ of $\DM_\cdh(K,R)$.
Then, for any object $N$ of $\DM_\cdh(K^s,R)$, we have
$\Hom(N,u^*(C))=0$: indeed, such an $N$ must be of the form
$u^*(M)$ for some $M$ in $\DM_\cdh(K,R)$, and the functor
$u^*$ is fully faithful on $\uDM_\cdh(-,R)$.
Since $K^s$ is a perfect field, this proves that $u^*(C)=0$,
and using the fully faithfulness of $u^*$ one last time
implies that $C=0$. This proves that $\DM_\cdh(K,R)=\uDM_\cdh(K,R)$
and achieves the proof of the proposition.
\end{proof}

\begin{cor}\label{cor:cancel}
Let $K$ be a field of exponent characteristic $p$.
Then the infinite suspension functor
$$\Sigma^\infty:\uDMe_\cdh(K,R)\to\uDM_\cdh(K,R)=\DM_\cdh(K,R)$$
is fully faithful.
\end{cor}

\begin{proof}
Let $K^s$ be the inseparable closure of $K$.
The functor
$$\Sigma^\infty:\uDMe_\cdh(K^s,R)\to\uDM_\cdh(K^s,R)=\DM_\cdh(K^s,R)$$
is fully faithful: this follows from the fact that the functor
$$\Sigma^\infty:\DMe(K^s,R)\to\DM(K^s,R)$$
is fully faithful (which is a reformulation of
Voevodsky's cancellation theorem \cite{cancel}) and from
assertions (b) and (c) in Proposition \ref{prop:insepclosureeff}.

Pulling back along the
map $u:\Spec(K^s)\to\Spec(K)$ induces an essentially
commutative diagram of the form
$$\xymatrix{
\uDMe_\cdh(K)\ar[r]^{\Sigma^\infty}\ar[d]_{u^*}&\uDM_\cdh(K)\ar[d]^{u^*}\ar@{=}[r]
&\DM_\cdh(K,R)\ar[d]^{u^*}\\
\uDMe_\cdh(K^s)\ar[r]^{\Sigma^\infty}&\uDM_\cdh(K^s)\ar@{=}[r]&\DM_\cdh(K^s,R)
}$$
and thus, Proposition \ref{prop:insepclosureeff} allows to conclude.
\end{proof}

\begin{num}\label{num:defbivmotcoh}
The preceding proposition and its corollary explain why it is essentially
harmless to only work with perfect ground fields\footnote{Note
however that the recent work of Suslin \cite{suslin} should
provide explicit formulas such as the one of Theorem \ref{thm:bivcycleDM6op}
for separated schemes of finite type over non-perfect infinite fields.}.
From now on, we will focus on our fixed perfect field $k$
of characteristic exponent $p$, and will work with separated $k$-schemes of finite type.

Let $X$ be a separated $k$-scheme of finite type and $r\geq 0$ an integer.
Let $\zequi(X,r)$ be the presheaf with transfers
of equidimensional relative cycles of dimension $r$ over $k$
(see \cite[Chap.~2, page 36]{FSV}); its evaluation
at a smooth $k$-scheme $U$ is the free group of cycles in $U\times X$
which are equidimensionnal of relative dimension $r$ over $k$;
see \cite[Chap.~2, Prop.~3.3.15]{FSV}. 
If $\Delta^\bullet$ denotes the usual cosimplicial $k$-scheme,
$$\Delta^n=\Spec\big(k[t_0,\ldots,t_n]/(\sum_i t_i=1)\big)\, ,$$
then, for any presheaf of ablian groups $F$, the Suslin
complex $\underline{C}_*(F)$ is the complex associated to the
simplicial presehaf of abelian groups $F((-)\times\Delta^\bullet)$.
Let $Y$ be another $k$-scheme of finite type.
After Friedlander and Voevodsky, for $r\geq 0$, the ($R$-linear)
\emph{bivariant cycle cohomology
of $Y$ with coefficients in cycles on $X$} is defined as
the following $\cdh$-hypercohomology groups:
\begin{equation}
A_{r,i}(Y,X)_R=H^{-i}_\cdh(Y,\underline{C}_*(\zequi(X,r))_\cdh\otimes^\derL R)\, .
\end{equation}\label{eq:def:bivcyclecoh}
Since $\ZZ(Y)$ is a compact object in the derived category of
$\cdh$-sheaves of abelian groups, we have a canonical isomorphism
\begin{equation}\label{eq:univcoef:bivcyclecoh}
\derR\Gamma(Y,\underline{C}_*(\zequi(X,r))_\cdh\otimes^\derL R)
\simeq\derR\Gamma(Y,\underline{C}_*(\zequi(X,r))_\cdh)\otimes^\derL R
\end{equation}
in the derived category of $R$-modules.
We also put $A_{r,i}(Y,X)_R=0$ for $r<0$.

Recall that, for any separated $k$-scheme of finite type $X$,
we have its motive $M(X)$ and its motive with compact support $M^c(X)$.
Seen in $\DM(k,R)$, they are the objects associated to the
presheaves with transfers $R(X)$ and $R^c(X)$ on smooth $k$-schemes:
for a smooth $k$-scheme $U$, $R(X)(U)$ (resp. $R^c(X)(U)$)
is the free $R$-module on the set of cycles in $U\times X$
which are finite (resp. quasi-finite) over $U$ and dominant over an irreducible
component of $U$. We will also denote by $M(X)$ and $M^c(X)$
the corresponding objects in $\DM_\cdh(k,R)$ through the
equivalence $\DM(k,R)\simeq\DM_\cdh(k,R)$.
\end{num}

\begin{thm}[Voevodsky, Kelly]\label{thm:bivcohDM}
For any integers $r,i\in\ZZ$,
there is a canonical isomorphism of $R$-modules
$$A_{r,i}(Y,X)_R\simeq\Hom_{\DM(k,R)}(M(Y)(r)[2r+i],M^c(X))\, .$$
\end{thm}

\begin{proof}
For $R=\ZZ$, in view of Voevodsky's cancellation theorem, this is a reformulation
of \cite[Chap.~5, Prop.4.2.3]{FSV}
in characteristic zero; the case where the exponent characteristic is $p$,
with $R=\ZZ[1/p]$, is proved by Kelly in \cite[Prop.~5.5.11]{kelly}.
This readily implies this formula for a general $\ZZ[1/p]$-algebra $R$
as ring of coefficients, using \eqref{eq:univcoef:bivcyclecoh}.
\end{proof}

\begin{rem}\label{rem:MYfunct}
Let $g:Y\to\Spec(k)$ be a separated morphism of finite type.
The pullback functor
\begin{equation}\label{eq:MYfunct1}
\derL g^*:\DM_\cdh(k,R)\to\DM_\cdh(Y,R)
\end{equation}
has a left adjoint
\begin{equation}\label{eq:MYfunct2}
\derL g_\sharp:\DM_\cdh(Y,R)\to\DM_\cdh(k,R)\, .
\end{equation}
Indeed, this is obviously true if we replace $\DM_\cdh(-,R)$
by $\uDM_\cdh(-,R)$. Since we have $\DM(k,R)\simeq\DM_\cdh(k,R)=\uDM_\cdh(k,R)$
(\ref{prop:insepclosureeff}~(c)), the restriction of
the functor
$$\derL g_\sharp:\uDM_\cdh(Y,R)\to\uDM_\cdh(k,R)$$
to $\DM_\cdh(Y,R)\subset\uDM_\cdh(Y,R)$ provides the
left adjoint of the pullback functor $\derL g^*$ in the
fibred category $\DM_\cdh(-,R)$. This construction does not only
provide a left adjoint, but also computes it: the motive of $Y$
is the image by this left adjoint of the constant motive on $Y$:
\begin{equation}\label{eq:MYfunct3}
M(Y)=\derL g_\sharp(R_Y)\, .
\end{equation}
We also deduce from this description of $\derL g_\sharp$ that,
for any object $M$ of $\DM_\cdh(k,R)$, we have a
canonical isomorphism
\begin{equation}\label{eq:MYfunct4}
\derR g_*\derL g^*(M)\simeq\derR\uHom_R(M(Y),M)
\end{equation}
(where $\uHom_R$ is the internal Hom of $\DM_\cdh(k,R)$): again, this
readily follows from the analogous formula in $\uDM_\cdh(-,R)$).

If we wite $z(X,r)$ for the $\cdh$-sheaf asociated to $\zequi(X,r)$
(which is compatible with the notations of Suslin and
Voevodsky, according to \cite[Chap.~2, Thm.~4.2.9]{FSV}),
we thus have another way of expressing the preceding theorem.
\end{rem}

\begin{cor}\label{cor:bivcohDM}
With the notations of Remark \ref{rem:MYfunct},
we have a canonical isomorphism of $R$-modules:
$$A_{r,i}(Y,X)_R\simeq
\Hom_{\DM_\cdh(Y,R)}(R_Y(r)[2r+i],\derL g^*(M^c(X)))\, .$$
\end{cor}

\begin{num}\label{num:bivcohDMbis}
The preceding corollary is not quite the most natural
way to express bivariant cycle cohomology $A_{r,i}(Y,X)$.
Keeping track of the notations of Remark \ref{rem:MYfunct},
we can see that there is a canonical isomorphism
\begin{equation}\label{num:bivcohDMbis1}
g_!\, g^!(R)\simeq M(Y)\, .
\end{equation}
Indeed, we have:
$$\derR\uHom_R(g_!g^!(R),R)=\derR g_*\derR\uHom_R(g^!(R),g^!(R))\, .$$
But Grothendieck-Verdier duality (\ref{thm:duality}) implies that
$$R_Y=\derR\uHom_R(g^!(R),g^!(R))\, ,$$
and thus \eqref{eq:MYfunct4} gives:
$$\derR\uHom_R(g_!g^!(R),R)\simeq\derR g_*\derL g^*(R)\simeq\derR\uHom_R(M(Y),R)\, .$$
Since the natural map
$$M\to\derR\uHom_R(\derR\uHom_R(M,R),R)$$
is invertible for any constructible motive $M$ in $\DM_\cdh(k,R)$,
we obtain
the identification \eqref{num:bivcohDMbis1} (note that $M(Y)$ is
constructible; see \cite[Lemma 5.5.2]{kelly}).
\end{num}

\begin{cor}\label{cor:bivcohDM2}
With the notations of Remark \ref{rem:MYfunct},
we have a canonical isomorphism of $R$-modules:
$$A_{r,i}(Y,X)_R\simeq
\Hom_{\DM_\cdh(Y,R)}(g^!(R)(r)[2r+i],g^!(M^c(X)))\, .$$
\end{cor}

\begin{num}\label{num:McXDM}
Let $f:X\to\Spec(k)$ be a separated morphism of finite type.
We want to describe $M^c(X)$ in terms of the six operations in $\DM_\cdh(-,R)$.
\end{num}

\begin{prop}\label{prop:McXDM}
With the notations of \ref{num:McXDM},
there are canonical isomorphisms
$$M^c(X)\simeq\derR f_*\, f^!(R)\simeq\derR\uHom_R(f_!(R_X),R)$$
in the triangulated category $\DM_\cdh(k,R)$.
\end{prop}

\begin{proof}
If $f$ is proper, then $f_!(R_X)=\derR f_*(R_X)$, while $M^c(X)=M(X)$
(we really mean equality here, in both cases). Therefore,
we also have
$$\derR\uHom_R(M^c(X),R)=\derR\uHom_R(M(X),R)\simeq\derR f_*(R_X)=f_!(R_X)$$
in a rather canonical way: the identification $\derR\uHom_R(M(X),R)\simeq\derR f_*(R_X)$
can be constructed in $\uDM_\cdh(K,R)$, in which case it can be promoted
to a canonical weak equivakence at the level of the model category of
symmetric Tate spectra of complexes of ($R$-linear) $\cdh$-sheaves with transfers
over the category of separated $K$-schemes of finite type.
In particular, for any morphism $i:Z\to X$ with $g=fi$ proper, we have
a commutative diagram of the form
$$\xymatrix{
\derR\uHom_R(M(X),R)\ar[r]^(.6)\sim\ar[d]_{i^*}&\derR f_*(R_X)\ar[d]^{i^*}\\
\derR\uHom_R(M(Z),R)\ar[r]^(.6)\sim&\derR g_*(R_Z)
}$$
in the (stable model category underlying the) triangulated category $\DM_\cdh(X,R)$.

In the general case, let us choose
an open embedding $j:X\to\bar X$ with a proper $k$-scheme $q:\bar X\to\Spec(k)$,
such that $f=qj$. Let $\partial \bar X$ be a closed subscheme of $\bar X$ such that
$\bar X\setminus\partial \bar X$ is the image of $j$, and write $r:\partial\bar X\to\Spec(k)$
for the structural map. What precedes means that there is a canonical identification
between the homotopy fiber of the restriction map
$$\derR q_*(R_{\bar X})\to\derR r_*(R_{\partial\bar X})$$
and the homotopy fiber of the restriction map
$$\derR\uHom_R(M(\bar X),R)\to\derR\uHom_R(M(\partial\bar X),R)\, .$$
But, by definition of $f_!(R_X)$, and by virtue of \cite[Chap.~5, Prop.~4.1.5]{FSV}
in characteristic zero, and of \cite[Prop.~5.5.5]{kelly} in general,
this means that we have a canonical isomorphism
$$\derR\uHom_R(M^c(X),R)\simeq f_!(R_X)\, .$$
By duality (\ref{thm:duality}),
taking the dual of this identification leads to a canonical isomorphism
$\derR f_*\, f^!(R)\simeq M^c(X)$.
\end{proof}

\begin{thm}\label{thm:bivcycleDM6op}
Let $Y$ and $X$ be two separated $k$-schemes of finite type
with structural maps $g:Y\to\Spec(k)$ and $f:X\to\Spec(k)$.
Then, for any $r\geq 0$, there is a natural identification
$$A_{r,i}(Y,X)_R\simeq\Hom_{\DM_\cdh(k,R)}(g_!g^!(R)(r)[2r+i],\derR f_*f^!(R))\, .$$
\end{thm}

\begin{proof}
We simply put Corollary \ref{cor:bivcohDM2} and Proposition \ref{prop:McXDM}
together.
\end{proof}

\begin{cor}\label{cor:higherChow}
Let $X$ be an equidimensional quasi-projective $k$-scheme of dimension $n$,
with structural morphism $f:X\to\Spec(k)$, and consider any subring $\Lambda\subset\QQ$
in which the characteristic exponent of $k$ is invertible.
Then, for any integers $i$ and $j$, we have a natural isomorphism
$$\Hom_{\DM_\cdh(X,\Lambda)}(\Lambda_X(i)[j],f^!\Lambda)\simeq\mathrm{CH}^{n-i}(X,j-2i)\otimes\Lambda$$
(where $\mathrm{CH}^{n-i}(X,j-2i)$ is Bloch's higher Chow group.
\end{cor}

\begin{proof}
In the case where $k$ is of characteristic zero, this
is a reformulation of the preceding theorem and of
\cite[Chap.~5, Prop.~4.2.9]{FSV}.
For the proof of \emph{loc. cit.} to hold \emph{mutatis mutandis}
for any perfect field $k$ of characteristic $p>0$ (and with $\ZZ[1/p]$-linear
coefficients), we see that
apart from Proposition \ref{prop:insepclosureeff}
and Theorem \ref{thm:bivcohDM} above, the only ingredient
that we need is the $\ZZ[1/p]$-linear version of
\cite[Theorem 4.2.2]{FSV}, which is provided by results
of Kelly \cite[Theorems 5.4.19 and 5.4.21]{kelly}.
\end{proof}

\begin{cor}\label{cor:Chow}
Let $X$ be a separated $k$-scheme of finite type, with
structural morphism $f:X\to\Spec(k)$.
For any subring  $\Lambda\subset\QQ$ in which $p$ is invertible,
there is a natural isomorphism
$$\mathrm{CH}_n(X)\otimes\Lambda\simeq\Hom_{\DM_\cdh(X,\Lambda)}(\Lambda_X(n)[2n],f^!\Lambda)$$
for any integer $n$ (where $\mathrm{CH}_n(X)$ is the usual Chow group of
cycles of dimension $n$ on $X$, modulo rational equivalence).
\end{cor}

\begin{proof}
Thanks to \cite[Chap.~4, Theorem~4.2]{FSV} and to
\cite[Theorem 5.4.19]{kelly}, we know that
$$\mathrm{CH}_n(X)\otimes\Lambda\simeq A_{n,0}(\Spec(k),X)_\Lambda\, .$$
We thus conclude with Theorem \ref{thm:bivcycleDM6op} for
$r=n$ and $i=0$.
\end{proof}

\section{Realizations}

\begin{num}\label{num:realizationDMh}
Recall from paragraph \ref{num:generalized_DM} that,
for a noetherian scheme $X$,
and a ring a coefficients $\Lambda$, one can define
the $\Lambda$-linear triangulated category of mixed
motives over $X$ associated to the $\h$-topology
$\DM_{\h}(X,\Lambda)$. The latter construction is the subject
of the article \cite{CD4}, in which we see that $\DM_\h(X,\Lambda)$
is a suitable version of the theory of \'etale mixed motives.
In particular, we have a natural
functor induced by the $\h$-sheafification functor:
\begin{equation}\label{eq:hsheafDM}
\DM_\cdh(X,\Lambda)\to\DM_\h(X,\Lambda)\ ,
\quad M\mapsto M_\h\, .
\end{equation}
These functors are part of a premotivic adjunction in the sense of
\cite[Def.~1.4.6]{CD3}.

From now on, we assume that the schemes $X$ are defined over a given
field $k$ and that the characteristic exponent of $k$ is invertible in $\Lambda$.
Since both $\DM_\cdh$ and $\DM_\h$ are motivic categories
over $k$-schemes in the sense of \cite[Def.~2.4.45]{CD3}
(see Theorem \ref{thm:continuityabsDMcdh}
above and \cite[Theorem 5.6.2]{CD4},
respectively), we have the following formulas
(see \cite[Prop. 2.4.53]{CD3}):
\begin{align}
(M\otimes^\derL_\Lambda N)_\h&\simeq M_\h\otimes^\derL_\Lambda N_\h\\
(\derL f^*(M))_\h&\simeq\derL f^*(M_\h) \quad\text{(for any morphism $f$)}\\
(\derL f_\sharp(M))_\h&\simeq\derL f_\sharp(M_\h)
\quad\text{(for any smooth separated morphism $f$)}\\
(f_!(M))_\h&\simeq f_!(M_\h)
\quad\text{(for any separated morphism of finite type $f$)}\label{eq:hsheaff_!}
\end{align}
Note finally that the functor \eqref{eq:hsheafDM} has fully faithful
right adjoint; its essential image consists of objects of $\DM_\cdh$
which satisfy the property of cohomological $\h$-descent
(see \cite[Def.~3.2.5]{CD3}).
\end{num}

\begin{lm}\label{lm:hsheafcommutedirectimage1}
Let $f:X\to \Spec k$ be a separated morphism of finite type.
Then the natural morphism
$$(\derR f_*(\Lambda_X))_\h\to\derR f_*((\Lambda_X)_\h)$$
is invertible in $\DM_\h(k,\Lambda)$.
\end{lm}

\begin{proof}
We may assume that $k$ is a perfect field (using
Prop. \ref{prop:insepclosureeff}~(d)
as well as its analogue for the $\h$-topology (which
readily follows from \cite[Prop.~6.3.16]{CD4})).
We know that $\DM_\cdh(k,\Lambda)=\uDM_\cdh(k,\Lambda)$
by Prop. \ref{prop:insepclosureeff}~(c), and similarly
that $\DM_\h(k,\Lambda)=\uDM_\h(k,\Lambda)$ (since, by virtue of
de Jong's theorem of resolution of singularities by alterations,
locally for the $\h$-topology, any $k$-scheme of finite type is smooth).
The functor
$$\uDM_\cdh(k,\Lambda)\to\uDM_\h(k,\Lambda)\ , \quad  M\mapsto M_\h$$
is symmetric monoidal and sends $\derL f_\sharp(\Lambda_X)$ to
$\derL f_\sharp((\Lambda_X)_\h)$.
On the other hand, the motive $\derL f_\sharp(\Lambda_X)\simeq f_!(f^!(\Lambda))$
is constructible (see \eqref{num:bivcohDMbis1} for $g=f$ and Theorem \ref{thm:constructmotivic}),
whence has a strong dual in $\DM_\cdh(k,\Lambda)$
(since objects with a strong dual form a thick subcategory, this follows
from Proposition \ref{prop:generators}, by Poincar\'e duality;
see \cite[Theorems 2.4.42 and 2.4.50]{CD3}).
The functor $M\mapsto M_\h$ being symmetric monoidal, it preserves
the property of having a strong dual and preserves strong duals.
Since $\derR f_*(\Lambda_X)$ is the (strong) dual of $\derL f_\sharp(\Lambda_X)$
both in $\uDM_\cdh(k,\Lambda)$ and in $\uDM_\h(k,\Lambda)$, this proves this
lemma.
\end{proof}

\begin{lm}\label{lm:derRfcommutesumscdhandh}
Let $f:X\to Y$ be a $k$-morphism between separated
$k$-schemes of finite type. Then the functors
$$\derR f_*:\DM_\cdh(X,\Lambda)\to\DM_\cdh(Y,\Lambda)\quad
\text{and}\quad\derR f_*:\DM_\h(X,\Lambda)\to\DM_\h(Y,\Lambda)$$
commute with small sums.
\end{lm}

\begin{proof}
In the case of $\cdh$-motives follows from the fact that the functor
$$\derL f^*:\DM_\cdh(Y,\Lambda)\to\DM_\cdh(X,\Lambda)$$
sends a family of compact generators into a family of compact objects.
The case of $\h$-motives is proven in \cite[Prop.~5.5.10]{CD4}.
\end{proof}

\begin{prop}\label{prop:hsheafcommutedirectimage2}
Let $f:X\to Y$ be a $k$-morphism between separated
$k$-schemes of finite type. Then, for any object $M$
of $\DM_\cdh(X,\Lambda)$, the natural map
$$\derR f_*(M)_\h\to\derR f_*(M_\h)$$
is invertible in $\DM_\h(Y,\Lambda)$.
\end{prop}

\begin{proof}
The triangulated category $\DM_\cdh(X,\Lambda)$ is compactly generated by
objects of the form $\derR g_*(\Lambda_{X'}(n)$ for $g:X'\to X$ a proper
morphism and $n$ any integer; see \cite[Prop.~4.2.13]{CD3}, for instance.
Since the lemma is already known in the case of
proper maps (see equation \eqref{eq:hsheaff_!}),
we easily deduce from Lemma \ref{lm:derRfcommutesumscdhandh}
that we may assume $M$ to be isomorphic to the constant motive $\Lambda_X$.
In this case, we conclude with Lemma \ref{lm:hsheafcommutedirectimage1}.
\end{proof}

\begin{cor}\label{cor:hsheaf6operations}
Under the assumptions of paragraph \ref{num:realizationDMh}, the restriction
of the motivic functor
$M\mapsto M_\h$ \eqref{eq:hsheafDM}
to constructible objects commutes with the
six operations of Grothendieck over the category of separated $k$-schemes of finite
type.
\end{cor}

\begin{proof}
After Proposition \ref{prop:hsheafcommutedirectimage2}, we see that it is sufficient
to prove the compatibility with internal Hom and with operations of the form $g^!$ for any morphism $g$ between
separated $k$-schemes of finite type.

Let us prove that, for any separated $k$-scheme of finite type $Y$ and any
constructible objects $A$ and $N$ of $\DM_\cdh(Y,\Lambda)$, the natural map
$$\derR\uHom(A,N)_\h\to\derR\uHom(A_\h,N_\h)$$
is invertible in $\DM_\h(Y,\Lambda)$. We may assume that $A=f_\sharp(\Lambda_X)$
for some smooth morphism $f:X\to Y$. Since we have the canonical identification
$$\derR\uHom(\derL f_\sharp(\Lambda_X),N)\simeq\derR f_*\, f^*(N)\, ,$$
we conclude by using the isomorphism
provided by Proposition \ref{prop:hsheafcommutedirectimage2} in the case where
$M=f^*(N)$.

Consider now a separated morphism of finite type $f:X\to\Spec k$.
For any constructible objects $M$ and $N$ of $\DM_\cdh(X,\Lambda)$ and
$\DM_\cdh(k,\Lambda)$, respectively, we have:
\begin{align*}
\derR f_*(\derR\uHom(M_\h,f^!(N)_\h))
&\simeq\derR f_*(\derR\uHom(M,f^!(N))_\h)\\
&\simeq(\derR f_*\derR\uHom(M,f^!(N)))_\h\\
&\simeq\derR\uHom(f_!(M),N)_\h\\
&\simeq\derR\uHom(f_!(M_\h),N_\h)\\
&\simeq\derR f_*(\derR\uHom(M_\h,f^!(N_\h)))\, .
\end{align*}
Therefore, for any object $C$ of $\DM_\h(k,\Lambda)$, there is an isomorphism:
\begin{align*}
\derR\Hom(\derL f^*(C)\otimes^\derL_\Lambda M_\h,f^!(N)_\h)\simeq
\derR\Hom(\derL f^*(C)\otimes^\derL_\Lambda M_\h,f^!(N_\h))\, .
\end{align*}
Since the constructible objects of the form $M_\h$ are a generating family of
$\DM_\h(k,\Lambda)$, this proves that the natural map
$$f^!(N)_\h\to f^!(N_\h)$$
is an isomorphism.
The functor $M\mapsto M_\h$ preserves internal Hom's
of constructible objects, whence it follows
from Formula \eqref{eq:dualityfexcept0} that it preserves duality.
Therefore, Formula \eqref{eq:dualityfexcept1}
shows that it commutes with operations of the form $g^!$ for any morphism $g$ between
separated $k$-schemes of finite type.
\end{proof}

\begin{rem}
In the case where $\Lambda$ is of positive characteristic, the trianguated category
$\DM_\h(X,\Lambda)$ is canonically equivalent to the derived category
$\Der(X_\et,\Lambda)$
of the abelian category of sheaves of $\Lambda$-modules on the small \'etale
site of $X$; see \cite[Cor.~5.4.4]{CD4}. Therefore, Corollary \ref{cor:hsheaf6operations}
then provides a system of triangulated functors
$$\DM_\cdh(X,\Lambda)\to\Der(X_\et,\Lambda)$$
which preserve the six operations when restricted to constructible objects.
Moreover, constructible objects of
$\DM_\h(X,\Lambda)$ correspond to the full subcategory
$\Der^b_{\mathit{ctf}}(X_\et,\Lambda)$
of the category $\Der(X_\et,\Lambda)$
which consists of bounded complexes of sheaves of
$\Lambda$-modules over $X_\et$ with constructible
cohomology, and which are of finite tor-dimension;
see \cite[Cor.~5.5.4 (and Th.~6.3.11)]{CD4}.
Therefore, for $\ell\neq p$,
using \cite[Prop. 7.2.21]{CD4}, we easily get $\ell$-adic realizations which
are compatible with the six operations (on constructible objects)
over separated $k$-schemes of finite type:
$$\DM_{\cdh,c}(X,\ZZ[1/p])\to\Der^b_c(X_\et,\ZZ_\ell)\to
\Der^b_c(X_\et,\QQ_\ell)\, .$$
For instance, this gives an alternative proof of some of the
results of Olsson (such as \cite[Theorem 1.2]{Olsson1}).

Together with Theorem \ref{thm:bivcycleDM6op}, Corollary \ref{cor:hsheaf6operations}
is thus a rather functorial way to construct cycle class maps in \'etale
cohomology (and in any mixed Weil cohomology, since they define realization functors
of $\DM_\h(-,\QQ)$ which commute with the six operations on constructible objects; 
see \cite[17.2.5]{CD3} and \cite[Theorem 5.2.2]{CD4}). 
This provides a method to prove independence of $\ell$
results as follows. Let $X$ be a separated $k$-scheme of finite type, with
structural map $a:X\to\Spec k$, and $f:X\to X$
any $k$-morphism. Then $f$ induces an endomorphism of $\derR a_*(\ZZ[1/p]_X)$
in $\DM_\cdh(k,\ZZ[1/p])$. Since the latter object is constructible
(by Theorem \ref{thm:constructmotivic}~(a)), it has a strong dual
(as explained in the proof of Lemma \ref{lm:hsheafcommutedirectimage1}),
and thus one can define the trace of the morphism induced by $f$,
which is an element of $\ZZ[1/p]$ (since one can identify $\ZZ[1/p]$
with the ring of endomorphisms of the constant motive $\ZZ[1/p]$
in $\DM_\cdh(k,\ZZ[1/p])$ using Corollary \ref{cor:Chow}).
Let $\ell$ be a prime number distinct from the characteristic exponent
of $k$. Since the $\ell$-adic realization functor is symmetric monoidal,
it preserves the property  of having a strong dual and preserves traces
of endomorphisms of objects with strong duals. Therefore, if $\bar k$ is any choice of an
algebraic closure of $k$, and if $\bar X=\bar k\otimes_k X$, the
number
$$\sum_i(-1)^i\mathrm{Tr}\big [ f^*:H^i_\et(\bar X,\QQ_\ell)
\to H^i_\et(\bar X,\QQ_\ell) \big ]$$
is independent of $\ell$ and belongs to $\ZZ[1/p]$:
Corollary \ref{cor:hsheaf6operations} implies that it is the image through the
unique morphism of rings $\ZZ[1/p]\to\QQ_\ell$ of the trace of the
endomorphism of the motive $\derR a_*(\ZZ[1/p]_X)$ induced by $f$.
This might be compared with Olsson's proof in the case where $f$ is finite;
see \cite[Theorem 1.2]{Olsson2}.
One may also replace $H^i(\bar X,\QQ_\ell)$ with the evaluation at $X$ of
any mixed Weil cohomology defined on smooth $k$-schemes, and still use the
same argument.
\end{rem}

\begin{rem}
If the ring $\Lambda$ is a $\QQ$-algebra,
the functor $M\mapsto M_\h$ defines an equivalence of categories
$\DM_\cdh(X,\Lambda)\simeq\DM_\h(X,\Lambda)$ (so that Corollary \ref{cor:hsheaf6operations}
becomes a triviality).
This is because, under the extra hypothesis that $\QQ\subset\Lambda$,
the abelian categories of $\cdh$-sheaves of $\Lambda$-modules
with transfers and of $\h$-sheaves of $\Lambda$-modules are equivalent: by a limit
argument, it is sufficient to prove this when $X$ is excellent, and then, this is an exercise
which consists to put together \cite[Prop.~10.4.8, Prop.~10.5.8, Prop.~10.5.11 and Th.~3.3.30]{CD3}.
\end{rem}

\bibliographystyle{amsalpha}
\bibliography{DMcdh}
\end{document}